\documentclass[twoside, 11pt]{article}
\usepackage[english]{babel}
\usepackage[latin1]{inputenc}
\usepackage{amsmath,amssymb}
\usepackage{amsfonts}
\usepackage{latexsym}
\usepackage{graphics,graphicx,psfrag}
\usepackage{upref}
\usepackage{hyperref}
\usepackage{caption,accents, float}
\usepackage{cite}
\usepackage[dvipsnames]{color, xcolor}
\usepackage{changes}

\allowdisplaybreaks[1]

\numberwithin{equation}{section}

\newcommand{\ctL}{c_{\tilde{L}^{-1}}}

\newcommand{\calk}{\mathcal{K}}
\newcommand{\partt}{\frac{\partial}{\partial t}}

\newcommand{\supp}{\mathop{\mathrm{supp}\,}\nolimits}

\newcommand{\ds}{\displaystyle}
\newcommand{\real}{\mathbb{R}}

\newcommand{\nat}{\mathbb{N}}

\newcommand{\ud}{\mathrm{d}}

\newtheorem{lemma}{Lemma}[section]

\newtheorem{theorem}[lemma]{Theorem}
\newtheorem{corollary}[lemma]{Corollary}
\newtheorem{definition}[lemma]{Definition}
\newtheorem{assumption}[lemma]{Assumption}
\newtheorem{rem}[lemma]{Remark}
\newcommand{\remark}[1]{\begin{rem}{\upshape #1}\end{rem}}
\newtheorem{example}[lemma]{Example}

\newcommand{\proofstart}{\mbox{P\,r\,o\,o\,f\, :\quad}}
\newcommand{\proofend}{\nopagebreak\hfill\raisebox{0.3em}{\fbox{}}\\}
\newenvironment{proof}{\proofstart}{\proofend}

\newtheorem{problems}{Problem}

\newcommand{\D}{\mathcal{D}}
\newcommand{\V}{\mathcal{K}}

\newcommand{\bit}{\begin{itemize}}
\newcommand{\eit}{\end{itemize}}
\newcommand{\beq}{\begin{equation}}
\newcommand{\eeq}{\end{equation}}

\title{Besov regularity of non-linear parabolic PDEs on non-convex polyhedral domains }

\author{
Stephan Dahlke\footnote{Philipps-University Marburg,  FB12 Mathematics and Computer Science, Hans-Meerwein Stra\ss{}e, Lahnberge, 35032 Marburg, Germany. Email: \href{mailto:dahlke@mathematik.uni-marburg.de}{dahlke@mathematik.uni-marburg.de}}\ \thanks{The work of this author has been supported by Deutsche Forschungsgemeinschaft (DFG), Grant No. DA 360/22-1.} \ 
\qquad
Markus Hansen\footnote{Philipps-University Marburg,  FB12 Mathematics and Computer Science, Hans-Meerwein Stra\ss{}e, Lahnberge, 35032 Marburg, Germany. Email: \href{mailto:Markus.Hansen1@gmx.net}{Markus.Hansen1@gmx.net}}\ 
\qquad
Cornelia Schneider\footnote{\emph{Corresponding author}. Friedrich-Alexander University Erlangen-Nuremberg, Applied Mathematics III, Cauerstr. 11, 91058 Erlangen, Germany. Email: \href{mailto:cornelia.schneider@math.fau.de}{cornelia.schneider@math.fau.de}}\ \thanks{The work of this author has been supported by Deutsche Forschungsgemeinschaft (DFG), Grant No. SCHN 1509/1-2.\hfill 
\textcolor{black}{\today}
} 
}

\date{}

\begin{document}

\maketitle

\begin{abstract}
This paper is concerned with the regularity of  solutions to parabolic evolution equations. We consider semilinear problems on non-convex domains. 
Special attention is paid to the smoothness in the specific scale 
$\ B^r_{\tau,\tau}, \ \frac{1}{\tau}=\frac{r}{d}+\frac{1}{p}\ $ of Besov spaces.   The regularity in these spaces determines the approximation order
that can be achieved by adaptive and other nonlinear approximation schemes. We show that for all cases under consideration the Besov regularity is high enough to 
justify the use of adaptive algorithms. Our proofs are based on Schauder's fixed point theorem. 

{{\em Math Subject Classifications. Primary:}   35B65, 35K55, 46E35.   {\em Secondary:}  35A01, 46A03, 47H10, 65M12. }

{{\em Keywords and Phrases.} nonlinear parabolic PDEs, non-convex domains, Schauder's fixed point theorem, locally convex spaces,  Kondratiev spaces, Besov spaces, adaptivity.}

\end{abstract}

\tableofcontents

\section{Introduction}

In this paper we derive regularity estimates in weighted Sobolev and Besov spaces of solutions to
 evolution equations in non-smooth  domains of polyhedral type $D \subset \real^3$, cf. Definition \ref{standard}.
Precisely, we investigate   linear ($\varepsilon=0$) and nonlinear  ($\varepsilon>0$) equations of the form 
\begin{equation} \label{parab-1a-i}
\frac{\partial }{\partial t}u+L(x,D_x)u +\varepsilon u^{M}\ =\ f \quad    \text{ in }\  [0,T]\times D, 
\end{equation}
 with zero initial and Dirichlet boundary conditions, where  $m,M\in \nat$,  and $L$ denotes a uniformly elliptic differential  operator of order $2m$ with sufficiently smooth coefficients. 
 Special attention is paid to the spatial regularity of the
solutions to  (\ref{parab-1a-i})  in specific
non-standard  smoothness spaces, i.e., in the so-called {\em adaptivity
scale {of Besov spaces}}
\begin{equation} \label{adaptivityscale}
B^r_{\tau,\tau}(D), \quad \frac{1}{\tau}=\frac{r}{3}+\frac{1}{p}, \quad r>0.
\end{equation}

Our investigations are motivated by  fundamental questions arising in the context of the numerical treatment of equation  \eqref{parab-1a-i}. It is our goal to mathematically justify that  when solving parabolic PDEs on non-smooth domains  adaptive schemes can outperform non-adaptive ones. 
In an adaptive strategy, the choice of the underlying degrees of freedom is not a priori fixed but depends on the shape of the unknown solution. In particular, additional degrees of freedom are only spent in regions where the numerical approximation is still 'far away' from the exact solution.
Although the basic idea is convincing,  adaptive algorithms are hard to implement, so that beforehand a rigorous mathematical analysis to justify their use  is highly desirable. 
Since it has been shown that  the achievable order of adaptive algorithms depends on the regularity of the target function in the specific scale of Besov spaces \eqref{adaptivityscale}, whereas  on the other hand it is the regularity of the solution in the scale of Sobolev spaces, which encodes information on the convergence order for nonadaptive (uniform) methods,  we can draw the following conclusion: Adaptivity is justified,
 if the Besov regularity of the solution
in the Besov scale
(\ref{adaptivityscale}) is higher than its Sobolev smoothness!

For the case of {\em elliptic} partial differential equations, a lot of positive results in this direction are already established  \cite{Dah98, Dah99a, Dah99b,Dah02,DDD,DDHSW,Han15,HW18} inspired by the fundamental paper of Dahlke and DeVore \cite{DDV97}. It is well--known
that  if the domain under consideration, the right--hand side, and the coefficients are sufficiently smooth, then there 
is no reason why the  Besov smoothness should be higher than the Sobolev regularity \cite{ADN59}. However, on general Lipschitz domains and in particular in polyhedral domains, the situation changes
 completely: On these non-smooth domains, singularities at the boundary may occur
that diminish the Sobolev regularity of the solution significantly \cite{CW20,Cost19,JK95, Gris92, Gris11} (but
can be compensated by suitable weight functions).

To the best of our knowledge, not so many results in this direction are available for evolution equations 
so far.  First  results   for the special case of the heat equation have been reported 
in \cite{AGI08, AGI10, AG12}, but for a slightly different scale of Besov spaces.  
Inspired by these findings we studied parabolic equations of type \eqref{parab-1a-i} on polyhedral type domains  in the  paper \cite{DS20}  and its forerunner \cite{DS19} (on polyhedral cones).

 The results obtained in \cite{DS19,DS20} are very promising 
 and indicate that (as in the elliptic case) the appearing boundary singularities do not influence the Besov regularity too much. 
 However, as a drawback of the methods we used (Banach's fixed point theorem), the nonlinear regularity results obtained in \cite[Thms. 4.13, 5.6]{DS20} only hold for convex domains so far, cf. the explanations given in \cite[Rem.~3.8]{DS19}. The main reason for this were the restrictions we had to impose on the weight parameter $a$ appearing in the definition of the weighted Sobolev spaces $\mathcal{K}^m_{p,a}(D)$ (so-called Kondratiev spaces), which we used for studying  the regularity of the solution. In particular, the weight $\rho(x)^{a}$ involved  measures only the distance to the singular set of the domain (i.e., the edges and vertices of the polyhedral type domains) and it turned out that there is no suitable $a$ satisfying all of our requirements in case of a  non-convex domain.

 However, since on these  domains much more severe singularities might occur, we aim at removing this restriction on the parameter $a$. This will be done by imposing  stronger assumptions on the right--hand side $f$, requiring that it is  arbitrarily smooth with respect to time, i.e., 
 \[
 f\in \bigcap_{l=0}^{\infty}H^l([0,T], L_2(D)\cap \mathcal{K}^{\eta-2m}_{2,a}(D)). 
 \]
We already pursued this possibility in \cite{DS20} when studying the linear equation \eqref{parab-1a-i} with $\varepsilon=0$ and were able to weaken the restrictions on the parameter $a$ in order to allow a larger range. But unfortunately, the generalization of the results  to nonlinear problems was not straightforward, since the right hand sides are not taken from a Banach or a quasi Banach space. Therefore, 
we now invoke Schauder's fixed point theorem which allows us to work with  the coarse topology of the locally convex topological vector spaces from the right hand side. The price to pay for this is that we only get existence but not uniqueness of the solutions.

Let us summarize our results: In the linear case $\varepsilon=0$ it is known that if the right-hand side as well as its time derivatives are
contained in specific Kondratiev spaces, then  for every $t \in [0,T]$  the spatial Besov smoothness  
of the solution to \eqref{parab-1a-i} is always larger than  $2m$, provided that some technical conditions on the operator pencils are satisfied, see \cite{DS20}.
The reader should observe that the results are independent of the shape of the polyhedral domain, and that the classical Sobolev smoothness
in the extreme case might be limited by  $m$, see \cite{LL15}.  Therefore, for every $t$  the spatial Besov regularity can be  more than twice as
high as the Sobolev smoothness, which of course justifies the use of (spatial) adaptive algorithms. 
Moreover,  for smooth domains
and  right-hand sides in $L_2,$ the best one could  expect would be smoothness order $2m$ in 
the classical Sobolev scale. Hence, the Besov smoothness on polyhedral type domains is at least as high as the Sobolev smoothness 
on smooth domains. \\
In this paper we generalize this result to nonlinear parabolic  equations of the form \eqref{parab-1a-i}. We prove that in an intersection of  sufficiently small balls containing the solution of the corresponding linear equation, there exists
a solution to  
\eqref{parab-1a-i} possessing the same Besov smoothness in the scale 
\eqref{adaptivityscale}.  The proof is performed by a technically quite involved application of Schauder's fixed point theorem. 
The final result is stated in Theorem \ref{nonlin-B-reg3}. \\

In conclusion, the results presented in this paper imply that for each $t \in (0,T)$ the spatial Besov regularity of the unknown solutions of the problems studied here is much higher than the Sobolev regularity,  which justifies the use of spatial adaptive  algorithms.  This corresponds to the classical time-marching schemes such as the Rothe method. We refer e.g. to the monographs \cite{Lan01, Tho06} for a detailed discussion. Of course,  it would be tempting to employ adaptive strategies in the whole space-time cylinder.  First results in this direction have been reported in \cite{SS09}.  To justify also these schemes,  Besov regularity in the whole space-time cylinder has to be established. This case will be studied in a forthcoming paper.\\

\section{Preliminaries}
\label{app-not}

We collect some  notation used throughout the paper. As usual,  we denote by $\nat$ the set of all natural numbers, $\nat_0=\mathbb N\cup\{0\}$, and 
$\real^d$, $d\in\nat$,  the $d$-dimensional real Euclidean space with $|x|$, for $x\in\real^d$, denoting the Euclidean norm of $x$. 
By $\mathbb{Z}^d$ we denote the lattice of all points in $\real^d$ with integer components. 
For $a\in\real$, let  
$[a]$ denote its integer part and $a_+:=\max(a,0)$. \\
Moreover,  $c$ stands for a generic positive constant which is independent of the main parameters, but its value may change from line to line. 
The expression $A\lesssim B$ means that $ A \leq c\,B$. If $A \lesssim
B$ and $B\lesssim A$, then we write $A \sim B$.  

Given two quasi-Banach spaces $X$ and $Y$, we write $X\hookrightarrow Y$ if $X\subset Y$ and the natural embedding is bounded. By $\supp f$ we denote the support of the function $f$. For a domain $\Omega\subset \real^d$ and $r\in \nat\cup \{\infty\}$ we write $C^r(\Omega)$ for the space of all {real}-valued $r$-times continuously differentiable functions, 
whereas $C(\Omega)$ is the space of bounded uniformly continuous functions, and  $\mathcal{D}(\Omega)$ for the set of test functions, i.e., the collection of all infinitely differentiable functions defined on $\real^d$ with  compact support contained in $\Omega$. Moreover,  $L^1_{\text{loc}}(\Omega)$ denotes the space of locally integrable functions on $\Omega$. \\
For  a multi-index  $\alpha = (\alpha_1, \ldots,\alpha_d)\in \nat_0^d$ with  $|\alpha| := \alpha_1+\ldots+ \alpha_d=r$   and an $r$-times differentiable function $u:\Omega\rightarrow \real$, we write 
\[
D^{(\alpha)}u=\frac{\partial^{|\alpha|}}{\partial x_1^{\alpha_1}\dots \partial x_d^{\alpha_d}} u
\]
for the corresponding classical partial derivative as well as $u^{(k)}:=D^{(k)}u$ in the one-dimensional case. Hence, the space $C^r(\Omega)$ is normed by 
\[
\| u| C^r(\Omega)\|:=\max_{|\alpha|\leq r}\sup_{x\in \Omega}|D^{(\alpha)}u(x)|<\infty. 
\]
Moreover, $\mathcal{S}(\real^d)$ denotes the Schwartz space of rapidly decreasing functions. The set of distributions on $\Omega$ will be denoted by $\mathcal{D}'(\Omega)$, whereas $\mathcal{S}'(\real^d)$ denotes the set of tempered distributions on $\real^d$. The terms {\em distribution} and {\em generalized function} will be used synonymously. For the application of a distribution $u\in \mathcal{D}'(\Omega)$ to a test function $\varphi\in \mathcal{D}(\Omega)$ we write $(u,\varphi)$. The same notation will be used if $u\in \mathcal{S}'(\real^d)$ and $\varphi\in \mathcal{S}(\real^d)$ (and also for the inner product in $L_2(\Omega)$).  For $u\in \mathcal{D}'(\Omega)$  and a multi-index $\alpha = (\alpha_1, \ldots,\alpha_d)\in \nat_0^d$, we write $D^{\alpha}u$ for the $\alpha$-th {\em generalized} or {\em distributional derivative} of $u$ with respect to $x=(x_1,\ldots, x_d)\in \Omega$, i.e., $D^{\alpha}u$ is a distribution on $\Omega$, uniquely determined by the formula   
\[
(D^{\alpha}u,\varphi):=(-1)^{|\alpha|}(u,D^{(\alpha)}\varphi), \qquad \varphi \in \mathcal{D}(\Omega). 
\]
{In particular, if  $u\in L^1_{\text{loc}}(\Omega)$ and  there exists a function $v\in L^1_{\text{loc}}(\Omega)$ such that 
\[
\int_\Omega v(x)\varphi(x)\ud x=(-1)^{|\alpha|}\int_{\Omega}u(x)D^{(\alpha)}\varphi(x)\ud x \qquad \text{for all} \qquad \varphi \in \mathcal{D}(\Omega), 
\]
we say that $v$ is the {\em $\alpha$-th weak derivative} of $u$ and  write $D^{\alpha}u=v$. 
}
We also use the notation $
\frac{\partial^k}{\partial x_j^k}u:=D^{\beta}u
$ as well as $\partial_{x_j^k}:=D^{\beta}u$,   for some 
multi-index  $\beta=(0,\ldots, k, \ldots,0)$ with $\beta_j=k$, $k\in \nat$. Furthermore, for $m\in \nat_0$, we write $D^mu$ for any (generalized) $m$-th order derivative of $u$, where $D^0u:=u$ and $Du:=D^1u$. Sometimes we shall use subscripts such as $D^m_x$ or  $D^m_t $ to emphasize that we only take derivatives with respect to $x=(x_1, \ldots, x_d)\in \Omega$ or $t\in \real$.

\section{Sobolev, Kondratiev, and Besov spaces}

In this section we briefly collect the  basics  concerning   weighted and unweighted Sobolev  spaces as well as Besov spaces needed later on. In particular, we put $H^m=W^m_2$ and denote by $\mathring{H}^m$ the closure of test functions in $H^m$ and its  dual space by $H^{-m}$. 
By $H^m_{|\cdot|}$ we denote the spaces $H^m$ w.r.t. their half norms $|u|_{H^m}:=\|\partial_m u|L_2\|$ involving only the highest derivatives. We will work with these spaces later on due to technical reasons.

Moreover,  $\mathcal{C}^{k,\alpha}$, $k\in \nat_0$, stands for the usual H\"older spaces with exponent $\alpha\in (0,1]$. 
The following generalized version of Sobolev's embedding theorem for Banach-space valued functions will be useful, cf. \cite[Thm.~1.2.5]{co-habil}.

\begin{theorem}[{\bf Generalized Sobolev's embedding theorem}]\label{thm-sob-emb}\index{Sobolev's embedding theorem! generalized}
Let $1<p<\infty$,  $m\in \nat$, $I\subset \real$ be some bounded interval, and $X$ a Banach space. Then 
\begin{equation}
W^{m}_p(I,X)\hookrightarrow \mathcal{C}^{m-1,1-\frac 1p}(I,X). 
\end{equation}
\end{theorem}
Here the Banach-valued  Sobolev spaces are endowed with the  norm 
\[
\|u|W^{m}_p(I,X)\|^p:=\sum_{k=0}^m \|\partial_{t^k}u|L_p(I,X)\|^p \quad \text{with}\quad \|\partial_{t^k}u|L_p(I,X)\|^p:=\int_I \|\partial_{t^k}u(t)|X\|^p~\ud t,  
\]
whereas for the H\"older spaces we use 
\[
\|u|\mathcal{C}^{k,\alpha}(I,X)\|:=\|u|C^k(I,X)\|+|u^{(k)}|_{C^{\alpha}(I,X)},
\]
where $\|u|C^k(I,X)\|=\sum_{j=0}^k\max_{t\in I}\|u^{(j)}(t)|X\|$ and $|u^{(k)}|_{C^{\alpha}(I,X)}=\sup_{{s,t\in I,}\atop  {s\neq t}}\frac{\|u^{(k)}(t)-u^{(k)}(s)|X\|}{|t-s|^{\alpha}}$. \\

\medskip 

Later on, we shall use the following equivalent norm in $H^m(I,X)$. 
\begin{theorem}\label{thm-eq-norm}
Let $I\subset \real$ be some bounded interval,   $X$ be a Banach space, and $m\in \nat$. Then the  norm $\|\cdot|L_2([0,T],X)\|+|\cdot|_{H^{m}([0,T],X)}$ involving only the derivatives of highest order $m$ is equivalent to the standard norm in $H^m(I,X)$, i.e., it holds 
$$   \|\cdot| H^{m}([0,T],X)\|\sim_m \|\cdot|L_2([0,T],X)\|+|\cdot|_{H^{m}([0,T],X)},$$
where the appearing constants  may depend on $m$. 
\end{theorem}

\begin{proof}
Let  $u\in H^{m}([0,T],X)$. We only have to show that  
\begin{equation}\label{ineq-norm}  \|u| H^{m}([0,T],X)\|\leq c\left( \|u|L_2([0,T],X)\|+|u|_{H^{m}([0,T],X)}\right),
\end{equation}
since  the other direction is obvious. For this we indirectly assume that there is no constant $c>0$ such that  \eqref{ineq-norm} holds. This implies the existence of functions $u_n\in H^{m}([0,T],X)$, $n\in \nat$,  such that 
\[
\|u_n|H^{m}([0,T],X)\|> n \left(\|u_n|L_2([0,T],X)\|+|u_n|_{H^{m}([0,T],X)}\right). 
\]
We put  $\ds v_n:=\frac{u_n}{\|u_n|H^{m}([0,T],X)\|}$, which yields 
\[
\|v_n|H^{m}([0,T],X)\|=1\qquad \text{and}\qquad  \|v_n|L_2([0,T],X)\|+|v_n|_{H^{m}([0,T],X)}<\frac 1n.
\]
In particular, we see that $|v_n|_{H^{m}([0,T],X)}<\frac 1n$. Since $v_n$ is uniformly bounded in $H^{m}([0,T],X)$, from  the compactness of the embedding $H^{m}([0,T],X)\hookrightarrow H^{m-1}([0,T],X)$, we deduce the existence of a convergent subsequence $v_{n'}\in H^{m-1}([0,T],X)$, whose limit we denote by $v$, i.e., it holds 
\[
\|v-v_{n'}|H^{m-1}([0,T],X)\|\rightarrow 0\quad \text{as} \quad n'\rightarrow \infty. 
\]
As $v_{n'}\rightarrow v$ in $H^{m-1}$ we particularly have $v_{n'}\rightarrow v$ in $L_2$. Moreover, per assumption on $v_n$ we know
$$\|v_n|L_2([0,T],X)\|<\frac{1}{n}$$
for all $n\in\nat$, from which we conclude $\|v|L_2([0,T],X)\|=0$, thus $v=0$ a.e.
Then obviously also $v\in H^m([0,T],X)$, and we can further estimate
\begin{align*}
	\|v-v_{n'}|H^m([0,T],X)\|
		&=|v-v_{n'}|_{H^m([0,T],X)}+\|v-v_{n'}|H^{m-1}([0,T],X)\| \
		\rightarrow 0\quad \text{as} \quad n'\rightarrow \infty, 
\end{align*}
i.e. we conclude $v_{n'}\rightarrow 0$ in $H^m([0,T],X)$, which is a contradiction to the assumption
$\|v_n|H^m([0,T],X)\|=1$ for all $n\in\nat$.
\end{proof}

We now collect some notation for specific Banach-space valued Lebesgue and Sobolev spaces, which will be used when studying the regularity of solutions of parabolic PDEs. \\

Let   $\Omega_T:=[0,T]\times\Omega$. Then we abbreviate/identify    
\[
L_p(\Omega_T):=L_p([0,T], L_p(\Omega)).    
\]\label{extra-1}

Moreover, we put \label{extra-4}
\[
H^{l,m*}(\Omega_T):=H^{l-1}([0,T],\mathring{H}^m(\Omega))\cap H^l([0,T],H^{-m}(\Omega)) 
\]
normed by 
\[\|u|H^{l,m*}(\Omega_T)\|=\|u|H^{l-1}([0,T],\mathring{H}^m(\Omega))\|+\|u|H^l([0,T],H^{-m}(\Omega))\|.\]

\paragraph{Kondratiev spaces}

In the sequel we work to a great extent with   weighted Sobolev spaces,  the so-called {\em Kondratiev spaces} $\V^m_{p,a}(\mathcal{O})$,   defined as the collection of all  $u\in \mathcal{D}'(\mathcal{O})$, which have $m$ generalized derivatives satisfying 

\begin{equation}\label{Kondratiev-1}
\|u|\V^m_{p,a}(\mathcal{O})\|:=\left(\sum_{|\alpha|\leq m}\int_{\mathcal{O}} |\varrho(x)|^{p(|\alpha |-a)}|D^{\alpha}_x u(x)|^p\ud x\right)^{1/p}<\infty,
\end{equation}
where $a\in \real$, $1<p<\infty$, $m\in \nat_0$, $\alpha\in \nat^n_0$, and the weight function $\varrho: \mathcal{O}\rightarrow [0,1]$ is the smooth distance to the singular set of $\mathcal{O}$, i.e., $\varrho$ is a smooth function and in the vicinity of the singular set $S$  it is {equivalent} to the distance to that set.  Clearly, if $\mathcal{O}$ is a polygon in $\real^2$ or a  polyhedral domain in $\real^3$, then  the  singular set  $S$ consists of the vertices of the polygon or the vertices and edges of the polyhedra, respectively. If   $\mathcal{O}$   is bounded, the smooth distance $\varrho$ can be replaced by the usual (Lipschitz-continuous) distance $\rho(x)=\inf_{y\in S}|x-y|$. 

It follows directly from \eqref{Kondratiev-1} that the scale of Kondratiev spaces is monotone in $m$ and $a$, i.e., 
\beq\label{kondratiev-emb}
\V^m_{p,a}(\mathcal{O})\hookrightarrow \V^{m'}_{p,a}(\mathcal{O})\quad \text{and}\quad \V^m_{p,a}(\mathcal{O})\hookrightarrow \V^m_{p,a'}(\mathcal{O}),
\eeq 
if $m'<m$ and $a'<a$.

Moreover, we transfer the above concept to functions additionally depending on the time  $t\in [0,T]$:  We define Kondratiev type spaces,    denoted by $L_q((0,T),\V^m_{p,a}(\mathcal{O}))$, which   contain all functions $u(x,t)$ such that 
\begin{align}
\|u|&L_q((0,T), \V^m_{p,a}(\mathcal{O}))\|\notag\\
&:=\left(\int_{(0,T)}\left(\sum_{|\alpha|\leq m}\int_{\mathcal{O}} |\varrho(x)|^{p(|\alpha |-a)}|D^{\alpha}_x u(x,t)|^p\ud x\right)^{q/p}\ud t\right)^{1/q}<\infty, \label{Kondratiev-3}
\end{align}
with $0<q\leq \infty$ and  parameters $a,p,m$  as above.

\paragraph{Kondratiev spaces on domains of polyhedral type}
\label{domains}

For our analysis we make use of several properties of Kondratiev spaces that have been proved in \cite{DHSS20}. 
Therefore,   in our later considerations, we will mainly be interested in the case that   $\mathcal{O}$ is a bounded domain of polyhedral type.

The precise definition below is taken from  Maz'ya and Rossmann \cite[Def.~4.1.1]{MR10}.

\begin{definition}\label{standard}
A bounded domain $D\subset \real^3$ is defined to be  of polyhedral type if the following holds: \\
\begin{figure}[H]
\begin{minipage}{0.55\textwidth}
\begin{itemize}
\item[(a)] {\em The boundary $\partial D$ consists of smooth (of class $C^{\infty}$) open two-dimensional manifolds $\Gamma_j$ (the faces of $D$), $j=1,\ldots, n$, smooth curves $M_k$ (the edges), $k=1,\ldots, l$, and vertices $x^{(1)}, \ldots, x^{(l')}$. }
\item[(b)] {\em For every $\xi\in M_k$ there exists a neighbourhood $U_{\xi}$ and a $C^{\infty}$-diffeomorphism  $\kappa_{\xi}$ which maps $D\cap U_{\xi}$ onto $\mathcal{D}_{\xi}\cap B_1(0)$, where $\mathcal{D}_{\xi}\subset \real^3$ is a dihedron, which in polar coordinates can be described as }
 \end{itemize}
\end{minipage}\hfill 
\begin{minipage}{0.25\textwidth}
\includegraphics[width=4cm]{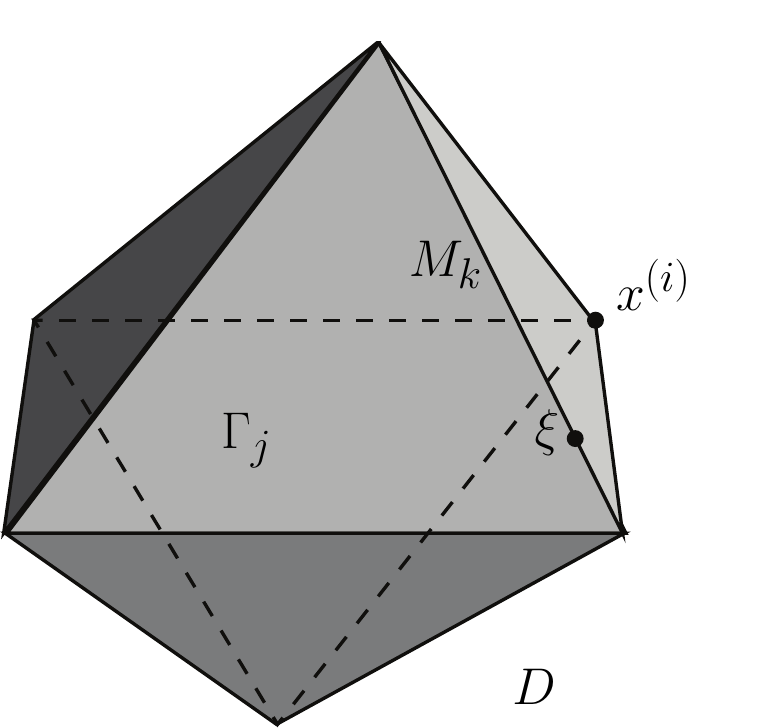}			
\caption[Polyhedron]{Polyhedron}
\end{minipage}\\
\begin{itemize}
    \item[] 
\[
\mathcal{D}_{\xi}=K\times \real, \qquad K=\{(x_1,x_2): \ 0<r<\infty, \ -\theta/2<\varphi<\theta/2\},
\]
{\em where the opening angle $\theta$ of the 2-dimensional wedge $K$ satisfies $0<\theta\leq 2\pi$. }
\item[(c)] {\em For every vertex $x^{(i)}$ there exists a neighbourhood $U_i$ and a diffeomorphism $\kappa_i$ mapping $D\cap U_i$
 onto $K_i\cap B_1(0)$, where $K_i$ is a polyhdral cone with edges and vertex at the origin. }
 \end{itemize}
\end{figure}

\end{definition}

\begin{rem}\label{rem-def-domain} \label{notation}
In the literature many different types of polyhedral domains are considered. A more general version which coincides with the above definition if  $d=3$ is discussed in \cite{DHSS20}.
Further variants of polyhedral domains can be found in 
Babu\v{s}ka,  Guo \cite{BG97},  Bacuta, Mazzucato, Nistor, Zikatanov \cite{BMNZ} 
and Mazzucato, Nistor \cite{NistorMazzucato}.
\end{rem}

\smallskip

Concerning pointwise multiplication  the following result can be found in \cite{DHSS20}.

\begin{corollary}\label{thm-pointwise-mult-2}
Let $\frac{d}{2}<p<\infty$, $m\in \nat_0$, and $a\geq \frac{d}{p}-2$. Then there exists a constant $c$ such that 
\[
\|uv| \mathcal{K}^{m}_{a,p}(D)\|\leq c\|u|\mathcal{K}^{m+2}_{a+2,p}(D)\|\cdot \|v|\mathcal{K}^{m}_{a,p}(D)\|
\]
holds for all $u\in \mathcal{K}^{m+2}_{a+2,p}(D)$ and $v\in \mathcal{K}^{m}_{a,p}(D)$.
\end{corollary}

\begin{rem}
In order to avoid any confusion we mention that in the proof of Theorem \ref{Sob-reg-3} (equation \eqref{est-3}) below, we  use the weaker estimate 
\[
\|uv| \mathcal{K}^{m}_{a,p}(D)\|\leq c\|u|\mathcal{K}^{m+2}_{a+2,p}(D)\|\cdot \|v|\mathcal{K}^{m+2}_{a+2,p}(D)\|. 
\]
\end{rem}

\paragraph{Besov spaces}

Due to the different contexts Besov spaces arose from they can be defined/characterized in several ways, e.g. via higher order differences, via the Fourier-analytic approach or via decompositions with suitable building blocks, cf. \cite{Tri83, Tri08} and the references therein.  Here we use the approach via higher order differences as can be found in \cite[Sect.~2.5.12]{Tri83}.   
Let $f$
be an arbitrary function on $\real^d$, $h\in\real^d$ and $r\in\nat$, then 
\[
(\Delta_h^1 f)(x)=f(x+h)-f(x) \quad \text{and} \quad
(\Delta_h^{r+1} f)(x)=\Delta_h^1(\Delta_h^r f)(x)
\]
are the usual iterated differences. 
Given a function $f\in L_p(\real^d)$ the {\em $r$-th modulus of smoothness} is defined
by
\[
\omega_r(f,t)_p:=\sup_{|h|\leq t} \|\Delta_h^r f\mid L_p(\real^d)\|, \quad
t>0. 
\]

Then the {\em Besov
space ${B}^s_{p,q}(\real^d)$} contains all $f\in L_p(\real^d)$ such that\label{besov-3}
\[
\|f|{B}^s_{p,q}(\real^d)\| := \|f|L_p(\real^d)\| + \left(\int_0^1 t^{-sq} \omega_r(f,t)_p^q
\ \frac{\ud t}{t}\right)^{1/q}<\infty, 
\]

where  $0<p,q\leq \infty$, $s>0$, and $r\in\nat$ such that $r>s$. 
This definition   is independent of $r$, meaning that different values
of $r>s$ result in quasi-norms which are equivalent.  Corresponding function spaces on domains $\mathcal{O}\subset \real^d$ can be introduced via restriction, i.e., 
\begin{eqnarray*}
B^s_{p,q}(\mathcal{O})&:=& \left\{f\in L_p(\mathcal{O}): \ \exists g\in B^s_{p,q}(\real^d), \ g\big|_{\mathcal{O}}=f \right\},\\
\|f|B^s_{p,q}(\mathcal{O})\|&:=& \inf_{g|_{\mathcal{O}}=f}\|f|B^s_{p,q}(\real^d)\|. 
\end{eqnarray*}

Our main tool when investigating the Besov regularity of  solutions to the PDEs will be the following embedding result  between  Kondratiev  and Besov spaces  adapted to our needs,   which is an extension of  \cite[Thm.~1]{Han15}. For a  proof we refer to \cite[Thms.~1.4.12, 1.4.14]{co-habil}.

\begin{theorem}[{\bf Embeddings between Kondratiev and Besov spaces}]\label{thm-hansen-gen}
Let $D\subset \real^3$ be some  polyhedral type domain and assume $k,m,\gamma\in \nat_0$. Furthermore, let $1<p<\infty$, $s, a\in \real$, and  suppose  $0\leq \gamma<\min(m,\frac{3}{\delta}s)$ and $a>\frac{\delta}{3}\gamma$, where $\delta$ denotes the dimension of the singular set (i.e., $\delta=0$ if there are only vertex singularities and $\delta=1$ if there are edge and vertex singularities). Then we have the continuous embedding
\begin{equation}\label{emb-hansen-gen-sob}
H^k([0,T],\calk^{m}_{p,a}(D))\cap H^k([0,T],B^s_{p,p}(D))\hookrightarrow H^k([0,T],B^{\gamma}_{\tau,\tau}(D)), \qquad \text{where}\quad \frac{1}{\tau}=\frac {\gamma}{3}+\frac 1p. 
\end{equation}
\end{theorem}

\remark{
Note that for the adaptivity scale of Besov spaces $B^{\gamma}_{\tau, \tau}(D)$ appearing in Theorem \ref{thm-hansen-gen}, 
from  the restriction on the parameters $\frac{1}{\tau}=\frac{\gamma}{3}+\frac 1p$ together with  $\tau\leq p$ and  $p> 1$, we deduce
\[
\gamma=3\left(\frac{1}{\tau}-\frac 1p\right)\geq 3\left(\frac{1}{\tau}-1\right)_+. 
\] 
In particular,  for this range of parameters it is well known that the different approaches towards Besov spaces (such as  the Fourier-analytic approach as well as decompositions via wavelets)  actually coincide. Therefore, the Besov spaces on domains (as before the   Kondratiev spaces) may  be considered in the setting of distributions, i.e., as subsets of $\mathcal{D}'(\mathcal{O})$, and  may  contain 'functions' which take complex values. However, when considering the fundamental parabolic problems, we restrict ourselves to the  real-valued setting: We assume the  coefficients of the differential operator $L$ to be real-valued as well as the right-hand side $f$, therefore, the solutions are real-valued as well. 
}

\section{Parabolic PDEs and operator pencils}

In this section we introduce the basic linear and nonlinear parabolic problems we will be concerned with in the sequel. Moreover, in order to state and prove our main results in Section \ref{sect-reg-sob-kon}, a short discussion of operator pencils is necessary.

\subsection{The fundamental parabolic problems}

Let  $D$  denote some domain of polyhedral type in $\real^d$ according to Definition \ref{standard} with faces $\Gamma_j$, $j=1,\ldots, n$.  
For $0<T<\infty$ put $D_T=(0,T]\times D$ and 
$ 
\Gamma_{j,T}=[0,T]\times \Gamma_j$.  \\

We will use the regularity results obtained in \cite{DS20} of the following linear  parabolic problem.

\begin{problems}[{\bf Linear parabolic problem in divergence form}]\label{prob_parab-1a}
Let $m\in \nat$. We consider the following first initial-boundary value problem 

\begin{equation} \label{parab-1a}
\left\{\begin{array}{rl}
\partt u+{L(x,D_x)}u\ =\ f \, &  \text{ in } D_T, \\
\frac{\partial^{k-1}u}{\partial \nu^{k-1}}\Big|_{\Gamma_{j,T}}\ =\ 0, & \   k=1,\ldots, m, \ j=1,\ldots, n,\\ 
u\big|_{t=0}\ =\ 0 \, & \text{ in } D.
\end{array} \right\}
\end{equation}
\end{problems}

Here $f$ is a function given on $D_T$, $\nu$ denotes the exterior normal to $\Gamma_{j}$, and  the partial differential operator $L$ is given by
\[{L(x,D_x)}=\sum_{|\alpha|, |\beta|=0}^m (-1)^{|\alpha|} D^{\alpha}_x({a_{\alpha \beta}(x)}D^{\beta}_x),\]
where $a_{\alpha \beta}$ are bounded real-valued functions from $C^{\infty}(D_T)$ with  $a_{\alpha \beta}={a}_{\beta \alpha}$ constant in $t$. 
Furthermore, the operator $L$ is assumed to be uniformly elliptic, i.e.,  
\begin{equation}\label{operator_L}
\sum_{|\alpha|, |\beta|=m}a_{\alpha \beta}\xi^{\alpha}\xi^{\beta}\geq c|\xi|^{2m} \qquad {\text{for all}}\quad  x\in D , \quad \xi\in \mathbb{R}^d.
\end{equation}

 Let us denote by 
\begin{equation}
B(t,u,v)=\int_D \sum_{|\alpha|, |\beta|=0}^m a_{\alpha\beta}(x)(D^{\beta}_xu) (D^{\alpha}_xv)\ud x
\end{equation}
the (time-dependent) bilinear form.

\begin{rem}[{\bf Assumptions on the time-dependent bilinear form}]\label{rem-B-coercive}
When dealing with parabolic problems  we can assume w.l.o.g.  that $B(t,\cdot, \cdot)$ satisfies 
\begin{equation}\label{B-coercive}
B(t,u,u)\geq \mu \|u|H^m(D)\|^2
\end{equation}
for all $u\in \mathring{H}^m(D)$ and a.e. $t\in [0,T]$. 
We refer to \cite[Rem.~2.3.5]{co-habil} for a detailed discussion. 
\end{rem}

It is our intention to also  study nonlinear  versions of Problem \ref{prob_parab-1a}. Therefore, we  modify \eqref{parab-1a} as  follows. 

\begin{problems}[{\bf Nonlinear parabolic problem in divergence form}]\label{prob_nonlin}
Let $m,M\in \nat$  and $\varepsilon>0$. We consider the following nonlinear parabolic problem 
\begin{equation} \label{parab-nonlin-1}
\left\{\begin{array}{rl}
\frac{\partial }{\partial t}u+L(x,D_x)u +\varepsilon u^{M}\ =\ f \, &  \text{ in } D_T, \\
\frac{\partial^{k-1}u}{\partial \nu^{k-1}}\Big|_{\Gamma_{j,T}}\ =\ 0, & \   k=1,\ldots, m, \ j=1,\ldots, n,\\ 
u\big|_{t=0}\ =\ 0 \, & \text{ in } D. 
\end{array} \right\}
\end{equation}
\end{problems}
The assumptions on $f$ and  the operator $L$ are as in Problem \ref{prob_parab-1a}. When we establish Besov regularity results for Problem \ref{prob_nonlin} we interpret  \eqref{parab-nonlin-1} as a fixed point problem and  show that the regularity estimates for Problem \ref{prob_parab-1a}   carry over to Problem \ref{prob_nonlin}, provided that $\varepsilon$ is sufficiently small.

\subsection{Operator pencils} 
\label{subsect-op-pen}

In order to correctly state the global regularity results in Kondratiev spaces for Problems \ref{prob_parab-1a} and \ref{prob_nonlin},   we need to work with operator pencils generated by  the corresponding elliptic problems in the polyhedral type domain ${D}\subset \real^3$.

We briefly recall the basic facts needed in the sequel. For further information  on this subject we refer to  \cite{KMR01} and  \cite[Sect.~2.3, 3.2., 4.1]{MR10}. On a domain 
 $D\subset\real^3$  of polyhedral type according to Definition \ref{standard} we consider the problem 
\begin{equation} 
\left\{\begin{array}{rl}\label{prob-ell-00}
Lu\ =\ f \, & \  \text{in} \quad D, \\
\frac{\partial^{k-1}u}{\partial \nu^{k-1}}\Big|_{\partial D}\ =\ 0, & \   k=1,\ldots, m.   
\end{array} \right\} 
\end{equation}

The singular set $S$ of $D$ then is given by the boundary points   $M_1\cup \ldots \cup M_l\cup \{x^{(1)}, \ldots, x^{(l')}\}$. We do not exclude the cases $l=0$ (corner domain) and $l'=0$ (edge domain). In the last case, the set $S$ consists only of smooth non-intersecting  edges. Figure  \ref{corner-edge-dom}  gives examples of polyhedral domains without edges or corners, respectively. \index{domain! corner domain}\index{domain! edge domain}\\  

\begin{figure}[H]
\begin{minipage}{\textwidth}
\begin{center}
\includegraphics[width=8cm]{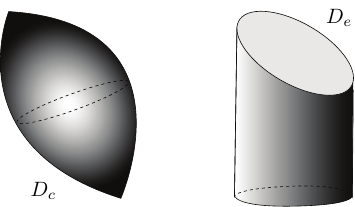}		
\caption[Corner and edge domain]{Corner domain $D_c$ ($l=0$) and edge domain $D_e$ ($l'=0$)}
\label{corner-edge-dom}
\end{center}
\end{minipage}\\
\end{figure}

The elliptic boundary value problem \eqref{prob-ell-00} on $D$ generates two types of operator pencils for the edges $M_k$ and for the vertices $x^{(i)}$ of the domain, respectively.

\begin{figure}[H]
\begin{minipage}{0.5\textwidth}
{\bf 1) Operator pencil $A_{\xi}(\lambda)$ for edge points:} \\
The pencils $A_{\xi}(\lambda)$ for  edge points $\xi\in M_k$ are defined as follows: According to Definition \ref{standard} there exists a neighborhood $U_{\xi}$ of $\xi$ and a diffeomorphism $\kappa_{\xi}$  mapping $D\cap U_{\xi}$ onto $\mathcal{D}_{\xi} \cap B_1(0)$, where $\mathcal{D}_{\xi}$ is a  dihedron. \\
Let $\Gamma_{k_{\pm}}$ be the faces adjacent to $M_k$. Then by $\mathcal{D}_{\xi}$ we denote the dihedron which is bounded by the half-planes $\mathring{\Gamma}_{k_{\pm}}$ tangent to $\Gamma_{k_{\pm}}$ at $\xi$ and the edge $M_{\xi}=\mathring{\Gamma}_{k_{+}}\cap \mathring{\Gamma}_{k_{-}}$. Furthermore, let $r,\varphi$ be polar coordinates in the plane perpendicular to $M_{\xi}$ such that 
\[
\mathring{\Gamma}_{k_{\pm}}=\left\{x\in \real^3: \ r>0, \ \varphi=\pm \frac{\theta_{\xi}}{2}\right\}.
\]
\end{minipage}\hfill \begin{minipage}{0.4\textwidth}
\includegraphics[width=6cm]{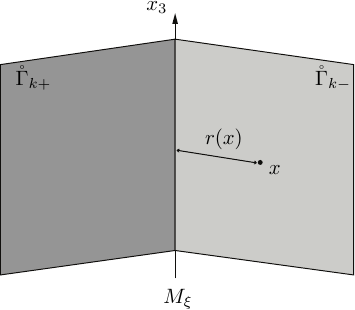}		
\caption[Dihedron $\mathcal{D}_{\xi}$]{Dihedron $\mathcal{D}_{\xi}$}
\end{minipage}\\
\end{figure}
We define the {\em operator pencil} $A_{\xi}(\lambda)$ as  follows:   \index{operator pencil! $A_{\xi}(\lambda)$}
\beq\label{op-pencil-1}
A_{\xi}(\lambda)U(\varphi)=r^{2m-\lambda}L_{0}(0,D_x)u, 
\eeq
where $u(x)=r^{\lambda}U(\varphi)$,  $\lambda\in \mathbb{C}$,  {$U$ is a function on $I_{\xi}:=\left(\frac{-\theta_{\xi}}{2}, \frac{\theta_{\xi}}{2}\right)$,}  and 
\[
L_{0}(\xi,D_x)=\sum_{|\alpha|=|\beta|=m}(-1)^{m}D^{\alpha}_x(a_{\alpha\beta}(\xi)D_x^{\beta}) 
\] 
denotes the main part of the differential operator $L(x,D_x)$ with coefficients frozen at $\xi$. 
This way we obtain in \eqref{op-pencil-1} a boundary value problem for the function $U$ on the 1-dimensional subdomain $I_{\xi}$ with the complex parameter $\lambda$.  Obviously,  $A_{\xi}(\lambda)$  is a polynomial of degree $2m$ in $\lambda$.

The operator $A_{\xi}(\lambda)$ realizes a continuous mapping 
\[
H^{2m}(I_{\xi})\rightarrow L_2(I_{\xi}),
\]
for every $\lambda\in \mathbb{C}$. 
Furthermore, $A_{\xi}(\lambda)$ is an isomorphism for all $\lambda\in \mathbb{C}$ with the possible exception of a countable set of isolated points, the {\em spectrum of  $A_{\xi}(\lambda)$}, \index{operator pencil! spectrum} which  consists of its eigenvalues with finite algebraic multiplicities: Here a  complex number $\lambda_0$ is called an {\em eigenvalue of the pencil $A_{\xi}(\lambda)$}\index{operator pencil! eigenvalue} if there exists  a nonzero function $U\in H^{2m}(I_{\xi})$ such that $A_{\xi}(\lambda_0)U=0$.
 It is known that the {\em 'energy line'} $\mathrm{Re}\lambda=m-1$  does not contain eigenvalues of the pencil $A_{\xi}(\lambda)$. We denote by {$\delta_{\pm}^{(\xi)}$} the largest positive real numbers such that the strip 
\beq\label{delta_k_op-1}
m-1-\delta_{-}^{(\xi)}<\mathrm{Re}\lambda<m-1+\delta_{+}^{(\xi)}
\eeq
is free of eigenvalues of the pencil $A_{\xi}(\lambda)$. Furthermore, we put 
\beq\label{delta_k_op}
{\delta_{\pm}^{(k)}}=\inf_{\xi\in M_k}{\delta_{\pm}^{(\xi)}}, \qquad k=1,\ldots, l. 
\eeq 

For example, concerning the Dirichlet problem for the Poisson equation 
on a domain $D\subset \real^3$  of polyhedral type, the eigenvalues of the pencil $A_{\xi}(\lambda)$  are given by 
\[
\lambda_k=k\pi/\theta_{\xi}, \qquad k=\pm1, \pm2, \ldots, 
\]
where   $\theta_{\xi}$ is the inner angle at the edge point $\xi$, cf. \cite[Ex. 2.5.2]{co-habil}. 
Therefore, the first positiv eigenvalue is $\lambda_1=\frac{\pi}{\theta_{\xi}}$ and we obtain $\delta_{\pm}=\frac{\pi}{\theta_{\xi}}$, cf. \cite[Ex. 2.5.1]{co-habil}.  \\

{\bf 2) Operator pencil $\mathfrak{A}_i(\lambda)$ for corner points:} \\ Let $x^{(i)}$ be a vertex of $D$. According to Definition \ref{standard} there exists a neighborhood $U_i$ of $x^{(i)}$ and a diffeomorphism $\kappa_i$  mapping $D\cap U_i$ onto $K_i\cap B_1(0)$ , where 
\[
K_i=\{x\in \real^3: \ x/|x|\in \Omega_i\}
\]
is a polyhedral cone with edges and vertex at the origin. W.l.o.g. we may assume that the Jacobian matrix $\kappa_i'(x)$ is equal to the identity matrix at the point $x^{(i)}$. We introduce spherical coordinates  $\rho=|x|$, $\omega=\frac{x}{|x|}$ in $K_i$ and define the operator pencil \index{operator pencil! $\mathfrak{A}_{i}(\lambda)$}
\begin{equation}\label{op-pencil}
\mathfrak{A}_i(\lambda)U(\omega)=\rho^{2m-\lambda}L_{0}(x^{(i)},D_x)u,
\end{equation}
where $u(x)=\rho^{\lambda}U(\omega)$ and $U\in \mathring{H}^{m}(\Omega_i)$ is a function on $\Omega_i$. An {\em eigenvalue of  $\mathfrak{A}_i(\lambda)$}\index{operator pencil! eigenvalue} is a complex number $\lambda_0$ such that 
$\mathfrak{A}_i(\lambda_0)U=0$ for some nonzero function $U\in\mathring{H}^{m}(\Omega_i)$. 
The operator $\mathfrak{A}_i(\lambda)$ realizes a continuous mapping 
\[
\mathring{H}^{m}(\Omega_i)\rightarrow H^{-m}(\Omega_i).
\]
Furthermore, it is  known that  $\mathfrak{A}_i(\lambda)$ is an isomorphism for all $\lambda\in \mathbb{C}$ with the possible exception of a countable set of isolated points. The mentioned countable set consists of eigenvalues with finite algebraic multiplicities. 
\begin{figure}[H]
\begin{minipage}{0.4\textwidth}
Moreover, the eigenvalues of $\mathfrak{A}_i(\lambda)$ are situated, except for finitely many, outside a double sector $|\mathrm{Re}\lambda|<\varepsilon |\mathrm{Im}\lambda|$ containing the imaginary axis, cf. \cite[Thm. 10.1.1]{KMR01}. In Figure \ref{eigenvalues-pencil} the situation is illustrated: Outside the double sector (yellow area) there are only finitely many eigenvalues of the operator pencil $\mathfrak{A}_i(\lambda)$.   \\
Dealing with regularity properties of solutions, we look for the widest strip in the $\lambda$-plane, free of eigenvalues and containing the {\em 'energy line'}  $ \mathrm{Re}\lambda=m-3/2,$ 
cf. Assumption \ref{assumptions}. From what was outlined above, information on the width of this strip is obtained from lower estimates for real parts of the eigenvalues situated over the energy line.   \\
\end{minipage}\hfill \begin{minipage}{0.5\textwidth}
\includegraphics[width=15cm]{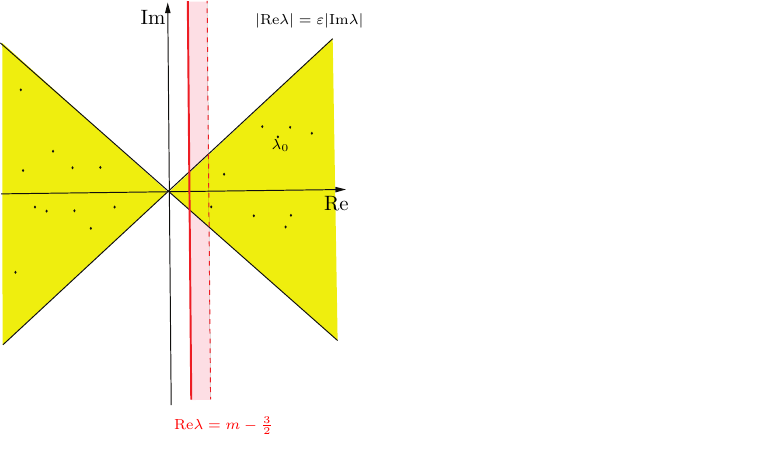}		
\caption[Eigenvalues of pencil $\mathfrak{A}_i(\lambda)$]{Eigenvalues of operator pencil $\mathfrak{A}_i(\lambda)$}
\label{eigenvalues-pencil}
\end{minipage}
\end{figure}

\begin{rem}[{\bf Operator pencils for parabolic problems}]  
Since  we study parabolic PDEs, where the differential operator $L(t,x,D_x)$ additionally depends on the time $t$, we have to work with operator pencils $A_{\xi}(\lambda,t)$ and $\mathfrak{A}_{i}(\lambda,t)$ in this context. The philosophy  is to fix $t\in [0,T]$ and define the pencils as above: We replace  \eqref{op-pencil-1} by 
\[
A_{\xi}(\lambda,t)U(\varphi)=r^{2m-\lambda}L_{0}(t,0,D_x)u, 
\]
and work with $\delta^{(\xi)}_{\pm}(t)$ and $\delta_{\pm}^{(k)}(t)=\inf_{\xi\in M_k}{\delta_{\pm}^{(\xi)}}(t)$ in \eqref{delta_k_op-1} and \eqref{delta_k_op}, respectively. Moreover, we put  
\beq\label{delta_k_op_t}
{\delta_{\pm}^{(k)}}=\inf_{t\in [0,T]}{\delta_{\pm}^{(k)}}(t), \qquad k=1,\ldots, l. 
\eeq 
Similar for $\mathfrak{A}_{i}(\lambda,t)$, where now \eqref{op-pencil} is replaced by 
\begin{equation}
\mathfrak{A}_i(\lambda,t)U(\omega)=\rho^{2m-\lambda}L_{0}(t,x^{(i)},D_x)u. 
\end{equation}
\end{rem}

\section{Regularity results in Sobolev and Kondratiev spaces}

\label{sect-reg-sob-kon}

This section presents regularity results for Problem \ref{prob_nonlin} in Sobolev and Kondratiev spaces based on the corresponding findings for Problem  \ref{prob_parab-1a} in \cite{DS20}. They will form the basis for obtaining regularity results in Besov  spaces later on via suitable embeddings.\\

\subsection{Regularity results in Sobolev and Kondratiev spaces for Problem I}
\label{subsect-sob-reg}

\begin{theorem}[{\bf Sobolev regularity with time derivatives}]\label{Sob-reg-3}
Let $l\in \nat_0$ and assume that the right hand side $f$ of Problem \ref{prob_parab-1a} satisfies 
$$f\in H^l([0,T], H^{-m}(D))\qquad \text{and} \qquad 
\partial_{t^k}f(0,x)=0  \quad  \text{ for } \quad  k=0,\ldots, l-1.$$ 
Then the weak solution $u$ in the space ${H}^{1,m*}(D_T)=L_2([0,T],\mathring{H}^m(\Omega))\cap H^1([0,T],H^{-m}(\Omega))$ of Problem \ref{prob_parab-1a} in fact belongs to ${H}^{l+1,m*}(D_T)$, i.e., 
has derivatives with respect to $t$ up to order $l$ satisfying 
\[
\partial_{t^k}u\in {H}^{1,m*}(D_T)\quad \text{for}\quad k=0,\ldots, l,
\]
and 
\[
\sum_{k=0}^l\|\partial_{t^k}u|H^{1,m*}(D_T)\| \leq C\sum_{k=0}^l\|\partial_{t^k} f|L_2([0,T],H^{-m}(D))\|, 
\]
where $C$ is a constant independent of $u$ and $f$. 
\end{theorem}

\begin{rem}
Note that the regularity results for the solution $u$ in  \cite[Thm.~2.1., Lem.~3.1]{LL15} are slightly stronger than the ones obtained in Theorem \ref{Sob-reg-3} above (with the cost of also assuming more regularity on the right hand side $f$).  
By using similar arguments as in  \cite[Lem.~4.3]{AH08} we are probably able to  show  
that if  $f\in L_2([0,T], L_2(D))$ then the weak solution $u$ of Problem \ref{prob_parab-1a} belongs in fact  to $L_2([0,T], \mathring{H}^m)\cap H^1([0,T], L_2(D))$ (instead of ${H}^{1,m*}(D_T)$). A corresponding generalization of Theorem \ref{Sob-reg-3} should also be possible in the spirit of \cite[Thm.~3.1]{AH08}. However, for our purposes the above results on the Sobolev regularity are sufficient, so these investigations are postponed for the time being. 
\end{rem}

We  need the following technical assumptions in order to state  the regularity results in Kondratiev spaces. 

\begin{assumption}[{\bf Assumptions on operator pencils}]\label{assumptions} 
Consider the operator pencils $\mathfrak{A}_i(\lambda,t)$,   $i=1,\ldots, l'$ for the vertices  and $A_{\xi}(\lambda,t)$ with $\xi\in M_k$, $k=1,\ldots, l$ for the edges of the polyhedral type domain $D\subset \real^3$ introduced in Section \ref{subsect-op-pen}.  Moreover,  we assume that $t\in[0,T]$ is fixed. \\
Let $\calk^{\gamma}_{p,b}(D)$ and $\calk^{\gamma'}_{p,b'}(D)$ be two Kondratiev spaces, where the singularity set $S$ of $D$ is given by $S=M_1\cup \ldots\cup M_l\cup \{x^{(1)},\ldots, x^{(l')}\}$ and weight parameters $b,b'\in \real$.   Then we  assume that the closed strip between the lines 
\beq\label{op-pen-ass1}
\mathrm{Re}\lambda=b+2m-\frac 32\qquad \text{and}\qquad \mathrm{Re}\lambda=b'+2m-\frac 32
\eeq 
does not contain eigenvalues of    $\mathfrak{A}_i(\lambda,t)$. Moreover, $b$ and $b'$ satisfy
\begin{equation}\label{restr-1-a}
-\delta_-^{(k)}<b+m<\delta_{+}^{(k)}, \qquad -\delta_-^{(k)}<b'+m<\delta_{+}^{(k)}, \quad k=1,\ldots, l, 
\end{equation}
where $\delta_{\pm}^{(k)}$ are defined  as in \eqref{delta_k_op_t}. 
\end{assumption}

\begin{rem} 
 If $l'=0$ we have an edge domain without vertices, cf.   Figure \ref{corner-edge-dom}. In this case condition  \eqref{op-pen-ass1} is empty. Moreover, if $l=0$, we have a corner  domain without edges, in which case condition \eqref{restr-1-a} is empty. 
 For further remarks and explanations  concerning Assumption \ref{assumptions} we refer to \cite[Rem.~3.3]{DS19}.
\end{rem}

\medskip

We recall the following regularity result from \cite[Thm.~4.11]{DS20}, which will be  discussed in Remark \ref{rem-weighted-sob-reg-2} below.

\begin{theorem}[{\bf Kondratiev regularity}]\label{thm-weighted-sob-reg-2}
Let $D\subset \real^3$ be a domain of polyhedral type and  $\eta\in \nat $ with  $\eta\geq 2m$. Moreover, let  $a\in \real$ with   $a\in [-m,m]$.  Assume that the right hand side $f$ of Problem \ref{prob_parab-1a} satisfies 
\begin{itemize}
\item[(i)] $ f\in \bigcap_{l=0}^{\infty}H^l([0,T], L_2(D)\cap \mathcal{K}^{\eta-2m}_{2,a}(D))$. 
\item[(ii)] $\partial_{t^l} f(0,x)=0$, \quad  $l\in \nat_0.$
\end{itemize}
Furthermore, let  Assumption \ref{assumptions}  hold for weight parameters $b=a$ and  $b'=-m$. 
Then for the weak solution $u\in  \bigcap_{l=0}^{\infty}{{H}}^{l+1,m\ast}(D_T)$ of Problem \ref{prob_parab-1a} we have 
$$\partial_{t^{l}} u\in L_2([0,T],\mathcal{K}^{\eta}_{2,a+2m}(D))\qquad \text{for all}\quad l\in \nat_0. $$
In particular, for the derivative $\partial_{t^l} u$  we have the a priori estimate 
\begin{align}
\sum_{k=0}^{l}& \|\partial_{t^{k}} u|{L_2([0,T],\mathcal{K}^{\eta}_{2,a+2m}(D))}\|\notag\\
&
\lesssim  \sum_{k=0}^{l+(\eta-2m)}\|\partial_{t^k} f|{L_2([0,T], \mathcal{K}^{\eta-2m}_{2,a}(D))}\|+\sum_{k=0}^{l+1+(\eta-2m)}\|\partial_{t^k} f|{L_2(D_T)}\|, \notag 
\end{align}
where the  constant is  independent of $u$ and $f$. 
\end{theorem}

\remark{\label{rem-weighted-sob-reg-2}
\bit  
\item We remark that the results from Theorems \ref{Sob-reg-3} and  \ref{thm-weighted-sob-reg-2} together yield that under the assumptions of Theorem \ref{thm-weighted-sob-reg-2} on the right hand side $f$  we have the following estimate for the solution $u=\tilde{L}^{-1}f$ of Problem \ref{prob_parab-1a}: 
\beq \label{apriori-est}
\|u|H^{l}([0,T],H^m(D)\cap\mathcal{K}^{\eta}_{2,a+2m}(D))\|  \lesssim_l  \|f| H^{l+1+\eta-2m}([0,T],L_2(D)\cap \mathcal{K}^{\eta-2m}_{2,a}(D))\|. 
\eeq 
Moreover, a careful inspection of the
proofs of Theorems \ref{Sob-reg-3} and  \ref{thm-weighted-sob-reg-2} (which may be found in \cite[Thms. 4.2.3, 4.2.12]{co-habil}) shows that for time independent coefficients  the critical constants
are {\em independent} of $l$ (which is why we restricted ouselves to this case here).
\item Due to   technical reasons, later on  we   work  with   the spaces  $H^l_{|\cdot|}([0,T],X)$, which stand for  $H^l([0,T],X)$ but only w.r.t. the  half norms $|u|_{H^l([0,T],X)}:=\|\partial_l u|L_2([0,T],X)\|$ involving the highest derivatives. Therefore, we wish to point out  that, assuming  $\eta=2m$ and   the coefficients of $\tilde{L}$ to be time independent, the a priori estimates from Theorems \ref{Sob-reg-3} and  \ref{thm-weighted-sob-reg-2} are  also valid when  considering the highest derivatives only, i.e., instead of   \eqref{apriori-est} we obtain  the substitute 
\begin{align} 
\|\partial_l u|L_2([0,T],H^m(D)\cap\mathcal{K}^{2m}_{2,a+2m}(D))\| & \lesssim_l  \|\partial_{l } f| L_2([0,T], \mathcal{K}^{0}_{2,a}(D))\|+\sum_{k=0}^1\|\partial_{l+k} f| L_2(D_T)\|\notag \\
& \lesssim_l \sum_{k=0}^1\|\partial_{l+k}f| L_2([0,T],L_2(D)\cap \mathcal{K}^{0}_{2,a}(D))\|.\label{apriori-est-1}
\end{align}
This follows from a close inspection of the proofs of Theorems \ref{Sob-reg-3} and  \ref{thm-weighted-sob-reg-2}: In particular, the proof of Theorem \ref{Sob-reg-3} requires that the coefficients need to be time independent. Moreover, the proof of Theorem \ref{thm-weighted-sob-reg-2} works for $\eta=2m$ only. Even for    $l=0$ we  otherwise get on the right hand  side of the a priori estimate  the    norm $\| f| L_2([0,T], \mathcal{K}^{\eta-2m}_{2,a}(D))\|$ in addition to the highest derivative $\| \partial_{\eta-2m}f| L_2([0,T], \mathcal{K}^{\eta-2m}_{2,a}(D))\|$,  which we wish to omit. 
\item Note that in \cite[Thm.~4.9, Thm.~4.11]{DS20} we established two different kinds of regularity results  in Kondratiev spaces of the linear Problem \ref{prob_parab-1a}. There we only used the results from \cite[Thm.~4.9]{DS20} when dealing with the nonlinear setting.  However, in order to now also deal with nonlinear problems on non-convex domains, we want to generalize the regularity results from  Theorem \ref{thm-weighted-sob-reg-2} for Problem \ref{prob_nonlin}:  
In Theorem \ref{thm-weighted-sob-reg-2} compared to \cite[Thm.~4.9]{DS20} we only require the parameter $a$ to satisfy $a\in [-m,m]$ and $-\delta_{-}^{(k)}<a+m<\delta_{+}^{(k)}$ independent of the regularity parameter {$\eta$} which can be arbitrarily high. In particular, for the heat equation on 
a  domain of polyhedral type $D$ (which for simplicity we assume to be  a polyhedron with straight edges and faces where   $\theta_k$ denotes the angle at the edge $M_k$), we have $\delta_{\pm}^{(k)}=\frac{\pi}{\theta_k}$, which leads to the restriction 
$ 
-1\leq  a<\min\left(1, \frac{\pi}{\theta_k}-1\right).  
$ 
Therefore,  even in the extremal case when $\theta_k=2\pi$ we can still take $-1\leq a<-\frac 12$ (resulting in $u\in L_2([0,T], \calk^{\eta}_{2,a+2}(D))$ being locally integrable since $a+2>0$). Then choosing  {$\eta$} arbitrary high,  we  also cover non-convex polyhedral type domains with our results from Theorem \ref{thm-weighted-sob-reg-2}.
\eit 
}

\subsection{Schauder's fixed point theorem}

We want to show that the regularity estimates in Kondratiev and Sobolev spaces of the linear Problem \ref{prob_parab-1a}  stated in Theorems     \ref{Sob-reg-3} and  \ref{thm-weighted-sob-reg-2}  carry over to Problem \ref{prob_nonlin}, provided that $\varepsilon$ is sufficiently small. In order to do this we interpret Problem \ref{prob_nonlin} as a fixed point problem in the following way. 

 Let $\mathcal{D}$ and $S$ be locally convex spaces ({$\mathcal{D}$ and $S$ will be specified in Theorem \ref{nonlin-B-reg1} below}) and let $\tilde{L}^{-1}:\mathcal{D}\rightarrow S$ be the linear operator defined  via
\begin{equation} \label{tildeL}
\tilde{L}u:=\frac{\partial}{\partial t}u+Lu. 
 \end{equation}
 Problem \ref{prob_nonlin}  is equivalent to 
\[
\tilde{L}u=f-\varepsilon u^{M}=:Nu,
\] 
where $N$ is a nonlinear operator. If we can show that $N$ maps $S$ into $\mathcal{D}$, then a solution of Problem \ref{prob_nonlin} is a fixed point of the problem 
\[
(\tilde{L}^{-1}\circ N)u=u.
\]

Our aim is  to apply  the following generalization of Schauder's Fixed Point Theorem from \cite[Thm.~A]{MN75}.

\begin{theorem}[Schauder's Fixed Point Theorem]
Let $S_0$ be a convex subset of a locally convex space $S$ and $T$ a continuous map of $S_0$ into a compact subset of $S_0$. Then $T$ has a fixed point.  
\end{theorem}

Therefore, Schauders's fixed point theorem will  guarantee the existence of a  solution, if we can show that  
\begin{equation} T:=(\tilde{L}^{-1}\circ N): S_0\rightarrow S_0\quad \text{is continuous}, \end{equation} 
where  $S_0\subset S$ is a  convex subset (which in our case is defined as an intersection of countably many closed balls    in the Sobolev spaces $H^k_{|\cdot|}([0,T],X)$, $k\in \nat$,   centered at the solution of the corresponding linear problem $\tilde{L}^{-1}f$ with suitably chosen radii $R_k>0$) and  
 \begin{equation}T(S_0)\subset {S_0} \quad \text{is compact}. \end{equation}

\subsection{Regularity results in Sobolev and Kondratiev spaces for Problem II}

Our main result is stated in the theorem below. Note that below we restrict ourselves to Kondratiev regularity $\eta = 2m$. We refer to the explanations given in Remark \ref{rem-weighted-sob-reg-2} in this context.

\begin{theorem}[{\bf Nonlinear Kondratiev regularity}]\label{nonlin-B-reg1}
Let $\tilde{L}$ and $N$ be as described above. Assume the assumptions of Theorem \ref{thm-weighted-sob-reg-2} are satisfied and, additionally, we have $m\geq  2$  and $a\geq - \frac 12$. Moreover, we put 
\[
\D_l:=H^l_{|.|}([0,T], L_2(D)\cap \mathcal{K}^{0}_{2,a}(D)), \quad 
\]
and consider the data space 
$$
\D :=\left\{f\in \bigcap_{l=0}^{\infty} \D_l : \  
\partial_{t^l} f(0,\cdot)=0, \quad l\in \nat_0\right\}.   
$$ 
Moreover, let 
\begin{align*}
S_k&:= H^{k}_{|.|}([0,T],\mathring{H}^m(D)\cap \mathcal{K}^{2m}_{2,a+2m}(D)), \qquad
\end{align*}
and  consider the solution space 
$$S:=\bigcap_{k=0}^{\infty}S_k.$$

Suppose that $f\in \D$ 
(which implies  $\tau_l:=|f|_{\D_l}<\infty$ for all $l\in \nat_0$). 
Then choosing $\varepsilon>0$ sufficiently small,  there exists a  solution $u\in S_0\subset S$ of Problem \ref{prob_nonlin}, where $S_0$ denotes the intersection of  balls  around $\tilde{L}^{-1}f$ (the solution of the corresponding linear problem)  with  radii  $R_k>0$,  i.e., 
\begin{equation}\label{def_S0}
S_0:=\bigcap_{k=0}^{\infty} B_{{k}}^{|\cdot|}(\tilde{L}^{-1}f,R_{k}),  
\end{equation}
where  $B_{{k}}^{|\cdot|}(\tilde{L}^{-1}f,R_{k })$ denotes the closed ball in $S_{k}$. 
\end{theorem}

\begin{proof}
{\em Step 1: \ Construction of compact sets $Q\subset S$.\  }
The solution space $$S=\bigcap_{k=0}^{\infty}H^{k}_{|.|}([0,T],X),\qquad X:=\mathcal{K}^{2m}_{2,a+2m}(D)\cap \mathring{H}^m(D),$$
is a locally convex Hausdorff space $(S, \{p_k\}_{k\in I})$ equipped with the  family of seminorms $$p_k(u):= |u|_{H^{k}([0,T],X)}, \quad k\in \nat_0.$$ 
Within the space $S$ we consider the set
\[
Q=\bigcap_{k=0}^{\infty}\overline{\mathrm{id}_{k}\left( B_{{k+1}}(\tilde{L}^{-1}f,\tilde{R}_{k+1})\right)},
\]
similar to \eqref{def_S0},  with   radii $\tilde{R}_{k+1}>0$ to be chosen later and $\mathrm{id}_{k}$ explained below. In this step we shall show that for any choice of the $\tilde{R}_{k+1}>0$ the set $Q$ is compact; therefore also
any closed subset $S_0\subset Q$ constructed in subsequent steps will be compact as well, cf. \cite[Thm. 1.29]{Dobro}. \\

\underline{\em Compactness of $Q$:} 
\bit 
\item We do not have the canonical compact embeddings within the Sobolev scale since  the  Sobolev spaces $H^k_{|.|}([0,T,X])$ are not monotonically ordered, i.e.
\mbox{$|\cdot|_{H^{k}([0,T],X)}\nleq |\cdot |_{H^{k+1}([0,T],X)}$}. Rescue comes from the fact that the embedding 
\[
\mathrm{id}_k: H^{k+1}([0,T],X)\hookrightarrow H^{k}([0,T],X)\hookrightarrow  {H}^{k}_{|.|}([0,T],X)\qquad \text{is compact.}
\]
This follows from the fact that the first embedding is compact, i.e., the completion of the image of the unit ball $B_{k+1}$ (and, consequently, of any bounded set) under the identity map $\mathrm{id}_k^{(1)}: H^{k+1}([0,T],X)\rightarrow H^{k}([0,T],X)$ is compact. Moreover, for the second embedding it is true  that $\mathrm{id}_k^{(2)}: H^{k}([0,T],X)\rightarrow {H}^{k}_{|.|}([0,T],X)$ is continuous and therefore the image of a compact set w.r.t. a continuous mapping is again compact. Thus, it remains to argue why 
\[
\mathrm{id}_k^{(2)}\Bigl(\overline{\mathrm{id}_k^{(1)}(B_{k+1})}\Bigr)
	=\overline{\mathrm{id}_k^{(2)}\Bigl(\mathrm{id}_k^{(1)}(B_{k+1})\Bigr)}
	=\overline{\mathrm{id}_k(B_{k+1})},  
\]
which ultimately shows that $\mathrm{id}_k=\mathrm{id}_k^{(2)}\circ \mathrm{id}_k^{(1)}$ is compact. 
However, this follows immediately since continuity is equal to sequential continuity in our spaces (see Appendix \ref{app-A}). 
More generally, we have $\overline{\mathrm{id}_k(B)}=\mathrm{id}_k\left(\overline{B}\right)$ for any ball $B\subset H^{k+1}([0,T],X)$.

\item Now let  
\begin{equation}\label{diag-sequence}
\{y_n^{(0)}\}_n\in Q
=\bigcap_{k=0}^{\infty}\overline{\mathrm{id}_{k}\left(B_{{k+1}}(\tilde{L}^{-1}f,\tilde{R}_{k+1})\right)}.
\end{equation}
Hence, $\{y_n^{(0)}\}_n$ belongs to the compact set $\overline{\mathrm{id}_{k}\left(B_{k+1}(\tilde{L}^{-1}f,\tilde{R}_{k+1})\right)}$ in ${H}^{k}_{|.|}([0,T],X)$ for any $k\in \nat_0$.
Therefore,  according to Remark \ref{comp-seqcomp} we find a subsequence $\{y_{n}^{(1)}\}_n$ of $\{y_{n}^{(0)}\}_n$ in the compact set
$\overline{\mathrm{id}_{k}\left(B_{1}(\tilde{L}^{-1}f,\tilde{R}_{1})\right)}$, which converges w.r.t. $\|\cdot|L_2(D)\|$ with limit in
$\overline{\mathrm{id}_{k}\left(B_{1}(\tilde{L}^{-1}f,\tilde{R}_{1})\right)}$,
\[
y_{n}^{(1)}\xrightarrow[]{n\rightarrow \infty} y^{(1)}\quad \text{in}\quad L_2([0,T],X).  
\]
Moreover, for any $k\in\nat$ we find a subsequence $\{y_n^{(k+1)}\}_n$
of $\{y_n^{(k)}\}_n$, which converges w.r.t. $|\cdot|_{H^{k}}$ with limit in $\overline{\mathrm{id}_{k}\left(B_{k+1}(\tilde{L}^{-1}f,\tilde{R}_{k+1})\right)}$.
\[
y_{n}^{(k+1)}\xrightarrow[]{n\rightarrow \infty} y^{(k+1)}\quad \text{in}\quad {H}^{k}_{|.|}([0,T],X).  
\]
In particular, from the above observations we obtain that 
\[
 y^{(k+1)}=\partial^k y^{(1)}\quad \text{in}\quad L_2([0,T],X), 
\]
which follows by the same standard arguments that one uses for proving the completeness of $H^k([0,T],X)$, cf. \cite[Thm.~5.10]{Dobro}. Hence, we deduce
$y_{n}^{(k+1)}\xrightarrow[]{n\rightarrow \infty} y^{(1)}$ in $H^k([0,T],X)$.  By continuing this procedure we see that this is true for all $k\in \nat_0$. 

Ultimately, we extract the diagonal sequence   $\{y^{\ast}_n\}_n$, given by $y^\ast_n=y_n^{(n)}$, 
satisfying 
\[
y^{\ast}_{n}\xrightarrow[]{n\rightarrow \infty}   y^{(1)} \quad \text{in}\quad {H}^k([0,T],X) \quad \text{for all} \quad k\in \nat_0. 
\]
In conclusion, the subsequence $\{y^{\ast}_n\}_n$ has a unique limit $y^{(1)}$,  which by our construction belongs to $\overline{\mathrm{id}_{k}\left( B_{{k+1}}(\tilde{L}^{-1}f,\tilde{R}_{k+1})\right)}$ for all $k\in \nat_0$, i.e., 
\begin{equation}\label{diag-sequence2}
y^{\ast}_n\xrightarrow[]{n\rightarrow \infty} y^{(1)}\in \bigcap_{k=0}^{\infty}\overline{\mathrm{id}_{k}\left(B_{{k+1}}(\tilde{L}^{-1}f,\tilde{R}_{k+1})\right)}=Q
\end{equation}
showing that $Q$ is  sequentially  compact. Since $Q\subset S$ and $S$ is a metrizable space, we deduce that this is equivalent to $Q$ being compact, cf. Appendix \ref{app-A}.
\eit

{\em Step 2: Definition of subsets $S_0\subset Q$.} Now define
$$S_0:=\bigcap_{k=0}^\infty{B_k^{|\cdot|}(\tilde{L}^{-1}f,R_k)}\,,$$
where the precise choice of the radii $R_k$, $k\in\nat_0$, will follow in subsequent steps. 
In particular, $S_0$ is closed. For this we need to show that for some sequence $\{y_n\}_n\in S_0$ with limit $y^{\ast}$ it follows that $y^{\ast}\in S_0$. This follows by a similar argumentation as in Step 1 above (see the construction below \eqref{diag-sequence}). We briefly sketch the proof. By the definition of $S_0$ and completeness of $L_2(D)$ it follows that $\{y_n\}_n$ converges w.r.t $\|\cdot|L_2(D)\|$ to $y^{\ast}\in B_0^{|\cdot|}(\tilde{L}^{-1}f,R_0)$, i.e., $y_n\xrightarrow{n\rightarrow \infty}y^{\ast}$ in $L_2([0,T],X)$. Moreover, for any $k\in \mathbb{N}$ we see that $\{y_n\}_n$ converges  to $y^{\ast}$ w.r.t $|\cdot|_{H^k}$
in $B_k^{|\cdot|}(\tilde{L}^{-1}f,R_k)$, i.e., $y_n\xrightarrow{n\rightarrow \infty}y^{\ast}$ in $H^k_{|\cdot|}([0,T],X)$. Hence, the unique limit of the sequence $\{y_n\}_n$ belongs to $B_k^{|\cdot|}(\tilde{L}^{-1}f,R_k)$ for all $k\in \mathbb{N}_0$ implying  $y^{\ast}\in S_0$. 
In this step, we shall show that for 
$\tilde{R}_{k+1}:=D_k(R_k+R_0)$, where $D_k$ is the constant from Theorem \ref{thm-eq-norm},  it holds 
$$S_0\subset Q=\bigcap_{k=0} ^{\infty}\overline{\mathrm{id}_{k}\left( B_{{k+1}}(\tilde{L}^{-1}f,\tilde{R}_{k+1})\right)}\,.$$

To this purpose, let $u\in S_0$. Then we have
$$\|u-\tilde{L}^{-1}f\|_{H^{k+1}}\leq D_k \left(|u-\tilde{L}^{-1}f|_{H^{k+1}}+|u-\tilde{L}^{-1}f|_{L_2}\right)
	\leq D_k(R_{k+1}+R_0)=\tilde{R}_{k+1}
 \,.$$
We conclude $u\in\overline{B_{k+1}(\tilde{L}^{-1}f,\tilde{R}_{k+1})}$ for all $k\in\nat_0$, and hence $u\in Q$.\\

{\em Step 3: \ (Continuity of $\tilde{L}^{-1}\circ N)$\  }

Let $u$ be the solution of the linear problem $\tilde{L}u=f$.  From Theorems  \ref{Sob-reg-3} and \ref{thm-weighted-sob-reg-2}, see also  \eqref{apriori-est-1},   we know   that 
$\tilde{L}^{-1}: \D_{l}\cup \D_{l+1}\rightarrow S_l   
$
is a bounded operator. If  $u^M\in {\mathcal{D}}$ (this will immediately follow from our calculations in Step 3  as explained in Step 4 below), the nonlinear part $N$ satisfies the desired mapping properties, i.e., $Nu=f-\varepsilon u^M\in {\mathcal{D}}$. 
Hence, 
we can apply Theorem  \ref{thm-weighted-sob-reg-2}  now with right hand side $Nu$. 
Moreover,  in view of the characterization of continuity in Fr\'echet spaces, the (nonlinear) map $\tilde{L}^{-1}\circ N:S_0\rightarrow S$ is continuous if we can show for all  $l\in \nat_0$   and $u,v\in S_0$ that  
\begin{equation}\label{est-0}
 | (\tilde{L}^{-1}\circ N)(u)-(\tilde{L}^{-1}\circ N)(v) |_{H^l([0,T],X)} \leq  c_l \sum_{k=l}^{l+1}|u-v|_{H^{k}([0,T],X)}, 
\end{equation}
where in this situation the constant $c_l$ may indeed depend on $l$. 
Since 
$
 (\tilde{L}^{-1}\circ N)(u)= \tilde{L}^{-1}(f-\varepsilon u^{M}) 
$ we therefore have to show 
 $$\varepsilon| \tilde{L}^{-1}(u^M-v^M)|_{H^l([0,T],X)}
	\lesssim_l  \sum_{k=l}^{l+1}|u-v|_{H^{k}([0,T],X)}. $$

We use Theorem \ref{thm-weighted-sob-reg-2} and the estimate \eqref{apriori-est-1}, which  yields 
\begin{align}
| \tilde{L}^{-1}(u^M-v^M)|_{H^l([0,T],X)}
&\lesssim_l  \ctL \sum_{i=0}^1|u^M-v^M|_{H^{l+i}([0,T],\mathcal{K}^{0}_{2,a}(D))} \notag \\
& \qquad  +\ctL \sum_{i=0}^1|u^M-v^M|_{H^{l+i}([0,T],L_2(D))}\notag \\
&=: I+II
\label{est-1}
\end{align}
Moreover, we make use of the formula $u^M-v^M=(u-v)\sum_{j=0}^{M-1}u^jv^{M-1-j}$. Concerning the derivatives,  applying Leibniz's formula twice  we see that 
\begin{align}
\partial_{t^k}(u^M-v^M)
&=\partial_{t^k}\left[(u-v)\sum_{j=0}^{M-1} u^jv^{M-1-j}\right]\notag\\
&= \sum_{w=0}^k{k \choose w}\partial_{t^w}(u-v)\cdot \partial_{t^{k-w}}\left(\sum_{j=0}^{M-1}u^j v^{M-1-j}\right)\notag\\
&= \sum_{w=0}^k{k \choose w}\partial_{t^w}(u-v)\cdot 
\left[\left(\sum_{j=0}^{M-1} \sum_{r=0}^{k-w} {{k-w}\choose r} \partial_{t^r}u^j \cdot \partial_{t^{k-w-r}}v^{M-1-j}\right)\right]. \notag\\\label{est-2}
\end{align}
In order to estimate the terms $\partial_{t^r}u^j$ (and similar for $\partial_{t^{k-w-r}}v^{M-1-j}$) we apply Fa\`{a} di Bruno's formula \index{Fa\`{a} di Bruno formula}
\begin{equation}\label{FaaDiBruno}
\partial_{t^r}(f\circ g)=\sum\frac{r!}{\kappa _1!\ldots \kappa_r!}\left(\partial_{t^{\kappa_1+\ldots + \kappa_r}}f\circ g\right)\prod_{{i}=1}^{r}\left(\frac{\partial_{t^{{i}}}g}{{i}!}\right)^{\kappa_{{i}}},
\end{equation}
where  the sum runs over all $r$-tuples of nonnegative integers $(\kappa_1,\ldots, \kappa_r)$ satisfying 
\begin{equation}\label{cond-kr}
1\cdot \kappa_1+2\cdot \kappa_2+\ldots +r\cdot \kappa_r=r.
\end{equation}
We apply the formula to $g=u$ and $f(x)=x^j$. 
In particular, from \eqref{cond-kr}  we see that $\kappa_{r}\leq 1$, where  $r=1,\ldots, k$. Therefore, the highest derivative $\partial_{t^r}u$ appears at most once.  Moreover, we 
  make use of the embeddings  \eqref{kondratiev-emb} and the pointwise multiplier result from Corollary \ref{thm-pointwise-mult-2} (note that this leads to our restriction $a\geq \frac d2-2=-\frac 12$ for  $d=3$). 
This yields

\begin{align}
\Big\|&\partial_{t^r}u^j  \left. | \mathcal{K}^{0}_{2,a}(D)\right\| \notag\\
&\leq  c_{r,j}\left\|\sum_{\kappa_1+\ldots +\kappa_r\leq j, \atop 1\cdot \kappa_1+2\cdot \kappa_2+\ldots +r\cdot \kappa_r=r} u^{j-(\kappa_1+\ldots +\kappa_r)}\prod_{{i}=1}^r \left|\partial_{t^{{i}}}u\right|^{\kappa_{{i}}}\Big| \mathcal{K}^{0}_{2,a}(D)\right\| \notag\\
&\lesssim \sum_{\kappa_1+\ldots +\kappa_r\leq j, \atop 1\cdot \kappa_1+2\cdot \kappa_2+\ldots +r\cdot \kappa_r=r} \left\| u| \mathcal{K}^{2}_{2,a+2}(D)\right\|^{j-(\kappa_1+\ldots +\kappa_r)}
\left\| \partial_{t^{{r}}}u| \mathcal{K}^{0}_{2,a}(D)\right\|^{\kappa_{{r}}}
\prod_{{i}=1}^{r-1} \left\| \partial_{t^{{i}}}u| \mathcal{K}^{2}_{2,a+2}(D)\right\|^{\kappa_{{i}}}\notag \\ 
&\lesssim \sum_{\kappa_1+\ldots +\kappa_r\leq j, \atop 1\cdot \kappa_1+2\cdot \kappa_2+\ldots +r\cdot \kappa_r=r} \left\| u| \mathcal{K}^{2m}_{2,a+2m}(D)\right\|^{j-(\kappa_1+\ldots +\kappa_r)} 
\prod_{{i}=1}^{r} \left\| \partial_{t^{{i}}}u| \mathcal{K}^{2m}_{2,a+2m}(D)\right\|^{\kappa_{{i}}}.  
\label{est-3}
\end{align}
For   $\partial_{t^{k-w-r}}v^{M-1-j}$ we obtain as in  \eqref{est-3}  with $\mathcal{K}^{0}_{2,a}(D)$ replaced by  $\mathcal{K}^{2}_{2,a+2}(D)$ the estimate 
\begin{align}\label{est-3aa}
\|&\partial_{t^{k-w-r}}v^{M-1-j}|\mathcal{K}^{2}_{2,a+2}(D)\| \notag\\
&\lesssim \sum_{\kappa_1+\ldots +\kappa_r\leq M-1-j, \atop 1\cdot \kappa_1+2\cdot \kappa_2+\ldots +(k-w-r)\cdot \kappa_{k-w-r}=k-w-r} \left\| v| \mathcal{K}^{2m}_{2,a+2m}(D)\right\|^{M-1-j-(\kappa_1+\ldots +\kappa_{k-w-r})} 
\prod_{{i}=1}^{k-w-r} \left\| \partial_{t^{{i}}}v| \mathcal{K}^{2m}_{2,a+2m}(D)\right\|^{\kappa_{{i}}}, 
\end{align} 
since $m\geq 2$. 
Now \eqref{est-2}, \eqref{est-3}, and \eqref{est-3aa}  
 together with  Corollary \ref{thm-pointwise-mult-2} give for the first term  in \eqref{est-1}: 
\begin{align}
I& = \ctL \sum_{k=l}^{l+1}|u^M-v^M|_{H^{k}([0,T],\mathcal{K}^{0}_{2,a}(D))}\notag \\
&= \ctL \sum_{k=l}^{l+1}\left(\int_0^T\left\|\partial_{t^k}\left[(u-v)\sum_{j=0}^{M-1} u^jv^{M-1-j}\right]\Big|\mathcal{K}^{0}_{2,a}(D)\right\|^2\ud t\right)^{1/2}\notag\\
& {
\lesssim \ctL \sum_{k=l}^{l+1}\sum_{w=0}^k\sum_{j=0}^{M-1}\sum_{r=0}^{k-w}\Bigg(\int_0^T \left\|\partial_{t^w}(u-v)
|\mathcal{K}^{0}_{2,a}(D)\right\|^2 }\notag \\
& {
\qquad \qquad 
\left\|\partial_{t^r}u^j |\mathcal{K}^{0}_{2,a}(D)\right\|^2
\left\|\partial_{t^{k-w-r}}v^{M-1-j}|\mathcal{K}^{2}_{2,a+2}(D)\right\|^2
\ud t\Bigg)^{1/2}
}\notag \\ 
& {
\lesssim \ctL \sum_{k=l}^{l+1}\sum_{w=0}^k\sum_{j=0}^{M-1}\sum_{r=0}^{k-w}\Bigg(\int_0^T \left\|\partial_{t^w}(u-v)
|\mathcal{K}^{0}_{2,a}(D)\right\|^2 }\notag \\
& {
\sum_{\kappa_1+\ldots+\kappa_r\leq j, \atop \kappa_1+2\kappa_2+\ldots+r\kappa_r=r}
\left\|u |\mathcal{K}^{2m}_{2,a+2m}(D)\right\|^{2(j-(\kappa_1+\ldots+\kappa_r))}
\prod_{i=0}^r \left\| \partial_{t^{{i}}}u| \mathcal{K}^{2m}_{2,a+2m}(D)\right\|^{2\kappa_{{i}}}
}\notag \\
& {
\sum_{\kappa_1+\ldots+\kappa_{k-w-r}\leq M-1-j, \atop \kappa_1+2\kappa_2+\ldots+(k-w-r)\kappa_{k-w-r}=k-w-r}
\left\|v |\mathcal{K}^{2m}_{2,a+2m}(D)\right\|^{2(M-1-j-(\kappa_1+\ldots+\kappa_{k-w-r}))} }\notag \\
&{\qquad \qquad 
\prod_{i=0}^{k-w-r} \left\| \partial_{t^{{i}}}v| \mathcal{K}^{2m}_{2,a+2m}(D)\right\|^{2\kappa_{{i}}}
\ud t\Bigg)^{1/2}
}\notag \\
& {
\lesssim \ctL \sum_{k=l}^{l+1}M\Bigg(\int_0^T \left\|\partial_{t^k}(u-v)
|\mathcal{K}^{0}_{2,a}(D)\right\|^2 }\notag \\
& {
\qquad 
\sum_{\kappa_1'+\ldots+\kappa_k'\leq \min\{M-1,k\}, \atop \kappa_k'\leq 1}
\max_{w\in \{u,v\}}\left\|w |\mathcal{K}^{2m}_{2,a+2m}(D)\right\|^{2(M-1-(\kappa_1'+\ldots+\kappa_k'))}}\notag \\
& {\qquad 
\qquad \prod_{i=0}^k  \max\left\{\left\| \partial_{t^{{i}}}u| \mathcal{K}^{2m}_{2,a+2m}(D)\right\|, \left\| \partial_{t^{{i}}}v| \mathcal{K}^{2m}_{2,a+2m}(D)\right\|, 1\right\}^{4\kappa_i'} \ud t\Bigg)^{1/2}
}\label{kappa}\\
&\lesssim M \ctL \cdot 
\sum_{k=l}^{l+1} |u-v|_{H^{k}([0,T],\mathcal{K}^{0}_{2,a}(D))}\cdot \notag\\
& \qquad \max_{w\in \{u,v\}}\max_{r=0,\ldots, l+1} \max \Big(\left\| \partial_{t^r}w |L_{\infty}([0,T],\mathcal{K}^{2m}_{2,a+2m}(D)\right\|,\  
1\Big)^{2({M-1})}.\label{est-4}
\end{align}
We give some explanations concerning the estimate above. 
In \eqref{kappa} we used the fact that in the step before we have two sums with  $\kappa_1+\ldots +\kappa_r\leq j$ and $\kappa_1+\ldots+\kappa_{k-w-r}\leq M-1-j$, i.e., we have $k-w$ different $\kappa_i$'s which leads to at most $k$ different $\kappa_i$'s if $w=0$.  We allow all combinations of $\kappa_i$'s and  redefine the $\kappa_i$'s in the second sum leading to $\kappa_1', \ldots , \kappa_k'$ with $\kappa_1'+\ldots+\kappa_k'\leq M-1$ and replace the old conditions $\kappa_1+\ldots +r\kappa_r\leq r$ and $\kappa_1+\ldots +(k-w-r)\kappa_{k-w-r}\leq k-w-r$ by the weaker ones $\kappa_1'+\ldots+\kappa_k'\leq k$ and $\kappa_k'\leq 1$. This causes no problems since the other terms appearing in this step do not depend on $\kappa_i$ apart from the product term. There, the fact that some of the old $\kappa_i$'s from both sums might coincide is reflected in the new exponent $4\kappa_i'$.   
From Theorem \ref{thm-sob-emb} (Sobolev's embedding theorem)  we conclude that

\begin{eqnarray}
u,v \in S&\hookrightarrow & {H}^{{r+1}}([0,T],X)
\hookrightarrow  {C}^{r}([0,T],X)
\hookrightarrow  {C}^{r}([0,T],\mathcal{K}^{2m}_{2,a+2m}(D)), \notag
\end{eqnarray}
hence the term  involving  $\max_{r=0,\ldots, l+1}(\ldots)^{2(M-1)}$ in \eqref{est-4} is bounded  by  
$$\max_{r=0,\ldots, l+1}(R_{r+1}+|\tilde{L}^{-1}f|_{H^{{r+1}}([0,T],X)}+\max_{w\in \{u,v\}}\|w|L_2([0,T],X)\|,1)^{2(M-1)},$$  since $u$ and $v$ are  taken from  $S_0=\bigcap_{k=0}^{\infty} \overline{B_k^{|\cdot|}(\tilde{L}^{-1}f,R_k)}$.

Thus   \eqref{est-4} yields    
 \begin{align}
I&=\ctL  \sum_{k=l}^{l+1} |u^M-v^M|_{H^{k}([0,T],\mathcal{K}^{0}_{2,a}(D))}\notag\\
&\lesssim M \ctL \max_{r=0,\ldots, l+1}\Bigl(R_{r+1}+|\tilde{L}^{-1}f|_{H^{{r+1}}([0,T],X)}+\max_{w\in \{u,v\}}\|w|L_2([0,T],X)\|,1\Bigr)^{{2(M-1)}} \sum_{k=l}^{l+1}|u-v|_{H^{{k}}([0,T],X)}. \label{est-ball}
\end{align}

We now turn our attention towards the second term 
in \eqref{est-1}. 
Altogether the whole calculation for $II$ is essentially the same as before:  We re-use Leibniz's formula from \eqref{est-2} and   Fa\`{a} di Bruno's formula from \eqref{FaaDiBruno}   in order to estimate the appearing derivatives. Moreover, by replacing $\mathcal{K}^{0}_{2,a}(D)$  by $L_2(D)$, $\mathcal{K}^{2m}_{2,a+2m}(D)$
 by $H^m(D)$, and Corollary \ref{thm-pointwise-mult-2}  by the estimate 
\[
\|uv|L_2(D)\|\leq \|u|C(D)\|\cdot \|v|L_2(D)\|\leq \|u|H^m(D)\|\cdot \|v|L_2(D)\|, 
\]
which is a simple consequence of Sobolev's embedding theorem (note that $m>\frac{3}{2}$ in view of the assumption on $m$), we ultimately obtain that 
 \begin{align}
II&=\ctL \sum_{k=l}^{l+1}|u^M-v^M|_{H^{k}([0,T],L_2(D))}\notag \\
& \lesssim 
\ctL M\max_{r=0,\ldots, l+1}\Bigl(R_{r+1}+|\tilde{L}^{-1}f|_{H^{{r+1}}([0,T],X)}+\max_{w\in \{u,v\}}\|w|L_2([0,T],X)\|,1\Bigr)^{2(M-1)} \notag \\
 & \qquad 
\cdot \sum_{k=l}^{l+1} |u-v|_{H^{k}([0,T],H^m(D))}.  \label{est-ball_a}
\end{align}
We shifted the details for \eqref{est-ball_a} to Appendix \ref{app-proof}  in order to make the presentation clearer.
Now \eqref{est-1} together with \eqref{est-ball} and \eqref{est-ball_a} yields 
\begin{align}\label{est-ab}
|& \tilde{L}^{-1}(u^M-v^M) |_{H^{l}([0,T],X)}\notag\\
 &\lesssim \ctL \sum_{k=l}^{l+1}|u^M-v^M|_{H^{l+1}([0,T],L_2(D)\cap \mathcal{K}^{0}_{2,a}(D))}\notag \\ 
& \lesssim M\ctL\cdot \sum_{k=l}^{l+1}|u-v|_{H^{k}([0,T],X)}\cdot
	\max\Bigl(R_{r+1}+|\tilde{L}^{-1}f|_{H^{{r+1}}([0,T],X)}+\max_{w\in \{u,v\}}\|w|L_2([0,T],X)\|,1\Bigr)^{2(M-1)}\notag\\  
& \lesssim M\ctL \cdot \sum_{k=l}^{l+1}|u-v|_{H^{k}([0,T],X)}\cdot
	\max\Bigl(R_{r+1}+\ctL (\tau_{r+1}+\tau_{r+2})+\max_{w\in \{u,v\}}\|w|L_2([0,T],X)\|,1\Bigr)^{2(M-1)}
\end{align}
where the maximum is taken over $r=0,\ldots, l+1$ and   the last step is a consequence of   \eqref{apriori-est-1} with  $\tau_k:=|f|_{H^{{k}}([0,T],L_2(D)\cap \mathcal{K}^{0}_{2,a}(D))}$, $k\in \nat$. This proves \eqref{est-0} and shows the continuity of  $\tilde{L}^{-1}\circ N: S_0\rightarrow S$. \\

{\em Step 4: } The calculations in Step 2 show that $u^M\in {\mathcal{D}}$: 
The fact that $u^M\in \bigcap_{k=0}^{\infty}\D_{k}$ for $u\in S$  follows from the estimate \eqref{est-ab}.   In particular, taking $v=0$ in \eqref{est-ab} we get an estimate from above for  $|u^M|_{\mathcal{D}_{l+1}}$. The upper bound depends on $|u|_{S_{l}}$, $|u|_{S_{l+1}}$,   and several constants which depend on $u$ but are finite whenever we have $u\in S$, see also \eqref{est-4} and \eqref{est-4a}. The dependence on the radii $R_k$, $k=1,\ldots, l+2$ in \eqref{est-ab} comes from the fact that we choose $u\in S_0$ there. 
However, the same argument can also be applied to an arbitrary $u\in S$; this would result in a different constant $\tilde{c}$. 
In order to have $u^M\in {\mathcal{D}}$, we still need to show that $\mathrm{Tr}\left(\partial_{t^k}u^M\right)=0$ for all  $k\in \nat_0$. This follows from similar arguments as  in   \cite[Thm.~4.13]{DS20}. \textcolor{black}{We omit the full details here and just briefly sketch the main idea}:  
Since $u\in S\hookrightarrow H^{k+1}([0,T], X)\hookrightarrow C^{k}([0,T], X)$ for any $k\in \nat_0$, we see that the trace operator $\mathrm{Tr}\left(\partial_{t^k}u\right):=\left(\partial_{t^k}u\right)(0,\cdot)$ 
is well defined for all $k\in \nat_0$ (independently on  $C^{0}([0,T],X)$).  
Using the initial assumption $u(0,\cdot)=0$ in Problem \ref{prob_nonlin}, 
by density arguments ($C^{\infty}(D_T)$ is dense in $S$)
and  induction we deduce  that $(\partial_{t^k}u)(0,\cdot)=0$ for all $k\in \nat_0$. 
 Moreover, since by Theorem \ref{thm-sob-emb} (Sobolev's embedding theorem) we have 
 \begin{align*}
u^M \in \bigcap_{l=0}^{\infty}\D_{l} &\hookrightarrow H^{k+1}([0,T], L_2(D)) 
\hookrightarrow C^{k}([0,T], L_2(D)), 
\end{align*}
 we see that also the trace operator $\mathrm{Tr}\left(\partial_{t^k}u^M\right):=\left(\partial_{t^k}u^M\right)(0,\cdot)$ is well defined for $k\in \nat_0$ (independently of $C^{0}([0,T],L_2(D))$). By \eqref{est-3a}  the term $\|\left(\partial_{t^k}u^M\right)(0,\cdot)|L_2(D)\|$ is estimated from above by {powers of} $\|\left(\partial_{t^l}u\right)(0,\cdot)|H^m(D)\|$, $l=0,\ldots, k$. 
 Since all these terms are equal to zero, 
 this shows that $u^M\in {\mathcal{D}}$. \\

{\em Step 5: \ } It remains to show that 
$
(\tilde{L}^{-1}\circ N)(S_0) \overset{!}{\subset} S_0$, which holds if  for  all $l\in \nat$  
we have 
\[
|(\tilde{L}^{-1}\circ N)(u)-\tilde{L}^{-1}f|_{H^{l}([0,T],X)}=\varepsilon | \tilde{L}^{-1}u^M|_{H^{l}([0,T],X)}\overset{!}{\leq }R_l. 
\]

We use the fact that  $u\in S_0$ implies 
\begin{equation}\label{cond-Rk} 
|u|_{H^{l}([0,T],X)}\leq R_l+ \ctL \tau_l,\qquad \text{for all } l\in \nat_0. 
\end{equation}
Let $Y=L_2(D)\cap \mathcal{K}^{0}_{2,a}(D)$ and $X=\mathring{H}^m(D)\cap \mathcal{K}^{2m}_{2,a+2m}(D)$. Then 
using \eqref{apriori-est-1} (which requires $\eta=2m$) and the fact that $\|\cdot|L_2([0,T],X)\|+|\cdot|_{H^{l}([0,T],X)}\sim_l \|\cdot| H^{l}([0,T],X)\|$ (where the constants appearing in the equivalence depend on $l$, see Theorem \ref{thm-eq-norm}) yields 

\begin{align}
\varepsilon &| \tilde{L}^{-1}u^M|_{H^{l}([0,T],X)}\notag \\
&\leq \varepsilon \ctL \left[ |u^M|_{H^{l+1}([0,T],Y)}+|u^M|_{H^{l}([0,T],Y)}\right]\notag\\
&\leq \varepsilon \ctL \left[ |u^M|_{H^{l+1}([0,T],X)}+|u^M|_{H^{l}([0,T],X)}\right]\notag\\
& \sim  \varepsilon \ctL \left[ (\|u^M |{H^{l+1}([0,T],X)}\|- \|u^M|L_2([0,T],X)\|)+(\|u^M| {H^{l}([0,T],X)}\|-  \|u^M|L_2([0,T],X)\|)\right]\notag\\
& \leq   2 \varepsilon \ctL \left[ (\|u^M |{H^{l+1}([0,T],X)}\|- \|u^M|L_2([0,T],X)\|)\right]\notag\\
&\leq   2 \varepsilon \ctL \left[ (c_l\|u  |{H^{l+1}([0,T],X)}\|^M- \|u^M|L_2([0,T],X)\|)\right]\notag\\
&\lesssim   2 \varepsilon \ctL \left[ c_l\big(|u  |_{H^{l+1}([0,T],X)}+\|u|L_2([0,T],X)\|\big)^M - \|u^M|L_2([0,T],X)\|\right].\label{est-p-1}
\end{align}
Here in the 5th step we used  the fact that  $H^k([0,T],X)=H^k([0,T],\mathring{H}^m(D)\cap \mathcal{K}^{\eta}_{2,a+2m}(D))$ is a multiplication algebra for any $k\geq \frac 12$ (this follows from \cite[Cor.~VII 6.2.4, Thm.~VII 7.3.4]{Ama19} together with \cite[Cor.~5.11]{DHSS20} and the fact that the interval $[0,T]$ has dimension $1$). 

 Thus to ensure the mapping properties for $\tilde{L}^{-1}\circ N$, combining \eqref{est-p-1} with \eqref{cond-Rk}, we require
\begin{align}
&\varepsilon | \tilde{L}^{-1}u^M|_{H^{l}([0,T],X)}\notag\\
&\lesssim 2\varepsilon \ctL c_l\left[  |u|_{H^{l+1}([0,T],X)} + \|u |L_2([0,T],X)\|^M\right] \notag\\
&\leq 2\varepsilon \ctL c_l\left[ (R_{l+1}+\ctL \tau_{l+1}) + (R_0+\ctL \tau_0)\right]^M\ \overset{!}{\leq} \ R_l. \label{estim-p-2} 
\end{align}

Now fix some arbitrary $N\in\nat$, and choose
$$R_l=\tau_l c_{\tilde L^{-1}}\quad\text{for all}\quad l\leq  N,\qquad
	R_l=0\quad \text{for all}\quad l> N\,.$$
Then for $l\leq N-1$ \eqref{estim-p-2}  simplifies to a recurrence relation for $R_l$,
$$R_l=(R_{l+1}+R_0)^M\cdot D(l)\,,$$
where we put $D(l):=2^{M+1}\varepsilon c_{\tilde L^{-1}}c_l$, together with
$$R_N=R_0^M D(N)\,.$$
This recurrence relation leads to
\begin{align*}
    R_{N-1}&=D(N-1)(R_N+R_0)^M\\
    &=D(N-1)\left(D(N)R_0^M+R_0\right)^M\\
    R_{N-2}&=D(N-2)(R_{N-1}+R_0)^M\\
    &=D(N-2)\left(D(N-1)(D(N)R_0^M+R_0)^M+R_0\right)^M\\
    &=:P_{N-2}(R_0),  \\
    & \ \ \vdots 
\end{align*}
where $P_{N-2}(R_0)$ is a polynomial in $R_0$ 
and ultimately
$$R_0=P_0(R_0)\,.$$
From this we see that $R_0$ has to be chosen as zero point of the polynomial 
\[
g(x):=P_0(x)-x.
\]
This is possible: Since   $g(0)=0$ and $g(x)\rightarrow \infty$ as $x\rightarrow \infty$ (all coefficients of $P_0$ are positive),  by choosing $\varepsilon>0$ small enough we can ensure $g'(0)<0$. From this it follows  that  there must exist some $R_0>0$ such that $g(R_0)=0$.  \\
In conclusion, from our procedure above    for any $N\in \nat$ chosen arbitrarily large, we obtain a well-defined sequence of radii $R_0, \ldots, R_N\geq 0$. 
\end{proof}

\remark{We are aware that the conditions in Theorem \ref{nonlin-B-reg1} above  are somewhat  restrictive. 
This is of course not ideal, however, this fact can not be avoided due to the proof techniques that we are using. Moreover, we believe that with the techniques used here, the results  obtained display the best possible outcome.  
In order to improve the results even further one would have to apply completely new proof techniques which to our knownledge are unknown so far. In particular, let us emphasize the following points in this context: 
\begin{itemize}
\item[(i)] The restriction $m\geq 2$ in Theorem \ref{nonlin-B-reg1} comes from the fact that we require $s=m>\frac d2=\frac 32$ in \eqref{multiplier-lim}. This assumption can probably be weakened, since we expect   the solution to satisfy  $u\in L_2([0,T], H^{s}(D))$ for all $s<\frac 32$, see also \cite[Rem.  5.3]{DS20} and the explanations given there.\\ 
Moreover, the restriction  $a\geq -\frac 12$ in  Theorem \ref{nonlin-B-reg1} comes from Theorem \ref{thm-pointwise-mult-2}(ii) that we applied. 
Together with the restriction $a\in [-m,m]$ we are looking for  $a\in [-\frac 12,m]$ if the domain is a corner domain, e.g.  a smooth cone $K\subset \real^3$ (subject to some truncation). For polyhedral cones with edges $M_k$, $k=1,\ldots, l$,  we   furthermore require $-\delta_-^{(k)}<a+m<\delta^{(k)}_+$  from Theorem \ref{thm-weighted-sob-reg-2}. 
\item[(ii)] Starting with \eqref{estim-p-2} and the same choice of $R_l$, i.e., $R_l=\tau_l \ctL$ for all $l\in \nat$, the more natural approach seems to be to solve \eqref{estim-p-2} for $R_{l+1}$, to get the estimate
\begin{equation}\label{estim-p-3}
    R_{l+1}\leq \frac 12\left[\frac{R_l}{2\varepsilon\ctL c_l}\right]^{1/M}-R_0\,,  
\end{equation}
i.e., a sequence of restrictions for $R_l$ (and thus in turn for $\tau_l=|f|_{\mathcal{D}_l}$) for all $l\in\nat$. The right-hand side of this condition needs to be strictly positive. However, choosing some sufficiently small $\varepsilon>0$, depending on (sufficiently small) $R_0$ cannot guarantee that the right-hand side of those conditions remains non-negative for all $l\in\nat$, as the precise dependencies of $c_{\tilde L^{-1}}$ and $c_l$ are not known. 
This in turn  means that in this case all higher seminorms (starting with the index $l$ corresponding to the first negative radius) would have to disappear and we end up with solutions that are polynomial perturbations of the solution to the linear system. However, such 
solutions can only exist for relatively short time intervals because  due to the smoothing property the
solutions  must decay, but polynomials do
grow beyond all limits. The result would then be relatively weak. \\
\end{itemize}
}

\section{Regularity results in Besov spaces}
\label{sect-besov-reg}

Based on the work done in Section \ref{sect-reg-sob-kon}, in this section we finally come to the presentation of the regularity results in Besov spaces for Problem \ref{prob_nonlin}. 
It turns out  that in all cases studied  the Besov regularity is higher than the Sobolev regularity (by factor $3$).  This indicates that adaptivity pays off when solving these problems numerically.  \\
Combining  Theorem \ref{nonlin-B-reg1} (Nonlinear  Kondratiev regularity) with the embeddings from Theorem \ref{thm-hansen-gen}  we derive the following result.

\begin{theorem}[{\bf Nonlinear Besov regularity}]\label{nonlin-B-reg3}
Let the assumptions of Theorems \ref{nonlin-B-reg1}  and  \ref{thm-weighted-sob-reg-2} be satisfied.  
In particular, as in Theorem \ref{nonlin-B-reg1} for $r_{\infty}:=\sup_{l\in \nat_0}\|f|\D_l\|<\infty$ and $r_0>1$,  we choose $\varepsilon >0$ so small that 
\begin{equation}\label{nonlin-cond1}
r_{\infty}^{M-1} \|\tilde{L}^{-1}\|^{M}\leq  \frac{r_0-1}{\varepsilon\cdot r_0^{M}}.  
\end{equation}
\begin{figure}[H]
\begin{minipage}{0.55\textwidth}
{\em  
Furthermore,  for $\alpha>0$ put 
$$B:=\bigcap_{k=0}^{\infty}H^k([0,T],B^{\alpha}_{\tau,\infty}(D)),   \qquad \frac{1}{\tau}=\frac{\alpha}{3}+\frac 12, 
$$
and define  $B_0$ as the intersection of small balls  around $\tilde{L}^{-1}f$ (the solution of the corresponding linear problem) with radius $R={C}(r_0-1)r_{\infty} \|\tilde{L}^{-1}\|>0$, i.e., 
\begin{equation}\label{def_B0}
B_0:=\bigcap_{k=0}^{\infty}\tilde{B}_{k+1}(\tilde{L}^{-1}f,R), 
\end{equation}
}
\end{minipage}\hfill \begin{minipage}{0.38\textwidth}
\includegraphics[width=8.5cm]{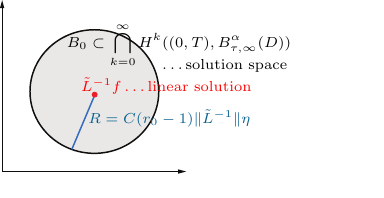}		
\caption[Nonlinear solution in  ball]{Nonlinear solution in $B_0$}
\end{minipage}
\end{figure}
where $\tilde{B}_{k+1}(\tilde{L}^{-1}f,R)$ denotes the closed ball in $H^k([0,T],B^{\alpha}_{\tau,\infty}(D))$. 
Then there exists a solution $u$ of Problem \ref{prob_nonlin}, which  satisfies 
 $$ u\in B_0\subset B \qquad 
\text{for all} \qquad 0<\alpha<\min\left(\frac{3}{\delta}m,\gamma\right),$$ 
where $\delta$ denotes the dimension of the singular set of $D$. 
\end{theorem}

\begin{proof}
This is a consequence of the regularity results in Kondratiev and Sobolev spaces from Theorem  \ref{nonlin-B-reg1}. To be more precise, Theorem  \ref{nonlin-B-reg1} establishes the existence of a fixed point $u$ in  
\begin{align*}
S_0\subset S
&:= \bigcap_{k=0}^{\infty}H^{k}([0,T],\mathring{H}^m(D)\cap\mathcal{K}^{\eta}_{2,a+2m}(D)).
\end{align*}
This together with the embedding results for Besov spaces from Theorem  \ref{thm-hansen-gen} 
completes the proof, since for all $k\in \nat$,  
\begin{align}
\| u&-\tilde{L}^{-1}f |H^k([0,T],B^{\alpha}_{\tau,\infty}(D))\|\notag \\
&\leq  C \| u- \tilde{L}^{-1}f| H^k([0,T], \mathcal{K}^{\eta}_{2,a+2m}(D)\cap H^m(D))\|
\leq  C (r_0-1)r_{\infty}\|\tilde{L}^{-1}\|=R. \label{est-abc}
\end{align} 
Furthermore,  it can be seen from \eqref{est-abc} that the new constant $C$ appears when considering the radius $R$ around the linear solution where the problem can be solved compared to Theorem \ref{nonlin-B-reg1}.
\end{proof}

\remark{A few words concerning the parameters appearing in Theorem \ref{nonlin-B-reg3} (and also Theorem  \ref{nonlin-B-reg1}) seem to be in order. 
Usually, the operator norm  $\|\tilde{L}^{-1}\|$ as well as  $\varepsilon$ are fixed; but we can change $r_{\infty}$ and $r_0$ according to our needs. From this we deduce that by choosing $r_{\infty}$ small enough the  {condition \eqref{nonlin-cond1} can always be satisfied.} Moreover, it is easy to see that the smaller the nonlinear perturbation $\varepsilon>0$ is, the larger we can choose the radius $R$ of the ball $B_0$. 
}

Finally, combining Theorem \ref{nonlin-B-reg3} with Theorem \ref{thm-sob-emb} gives the following result. 

\begin{corollary}[H\"older-Besov regularity]
Let the assumptions of Theorem \ref{nonlin-B-reg3} be satisfied and let $B_0$ be defined as in \eqref{def_B0}.  
Then there exists a solution $u$ of Problem \ref{prob_nonlin}, which  satisfies 
 $$ u\in B_0\subset C^{\infty}([0,T],B^{\alpha}_{\tau,\infty}(D))\qquad  
\text{for all} \qquad 0<\alpha<\min\left(\frac{3}{\delta}m,\gamma\right),$$ 
where $\delta$ denotes the dimension of the singular set of $D$. 
\end{corollary}

\appendix

\section{Locally convex spaces}\label{app-A}

 A locally convex topological vector space  (denoted as LC-space in the sequel) is defined to be a vector space $X$ along with a family $\mathcal{P}_0$ of seminorms on $X$. 
This family of seminorms  induces a canonical vector space topology on $X$, called the initial topology, turning it into a topological vector space. By definition, it is the coarsest topology on $X$ for which all maps in $\mathcal{P}_0$ are continuous.
In particular, the balls $B_p(x,r):=\{y\in X: \ p(x-y)<r\}$, where $x\in X$, $r>0$, and  $p\in \mathcal{P}_0$ constitute a subbasis of the topology.  
 Convergence in LC-spaces can be defined via the seminorms, i.e., a sequence $\{x_i\}_i$ converges in the underlying topology of a LC-space  towards $x$ if, and only if, $p(x_i-x)\xrightarrow{i\rightarrow \infty} 0$ for all $p\in \mathcal{P}_0$. 

In this paper we especially deal with the LC-spaces 
\[
X=\bigcap_{k=0}^{\infty}H^k_{|\cdot|}([0,T],\tilde{X}),  \qquad \tilde{X} \ \text{some quasi-Banach space}, 
\]
which are Hausdorffsch and for which the corresponding family of seminorms $p_k(\cdot)=|\cdot|_{H_k([0,T],\tilde{X})}$ is countable. 
They are  examples of  Fr\'echet-spaces since they are  complete (which follows from the fact that a Cauchy sequence in $X$ is a Cauchy sequence  w.r.t. any of the seminorms, thus, also in $H^k$ for any $k$; therefore,  by completeness of the  Sobolev spaces the sequence   converges for every $k$, and its unique limit again belongs to $X$).  Moreover, 
these kinds of spaces are metrizable by setting 
\[
d(x,y):=\sum_{k\in \nat}\frac{1}{2^k}\frac{p_k(x-y)}{1+p_k(x-y)}, \qquad x,y\in X.
\]
In particular, they satisfy the first axiom of countability, i.e., every point has a countable neighbourhood basis (local base).  We want to investigate nonlinear mappings on these spaces and are interested in their continuity. For general topological spaces continuous maps are defined via open neighbourhoods as follows: 
 Let $X$ and $Y$ be topological spaces and let $f:X\rightarrow Y$. 
Then $f$ is continuous if the pre-image $f^{-1}(V)$ of every open set $V$ in $Y$
 is an open subset of $X$. 
 Since our LC-spaces satisfy the first axiom of countability we additionally have the following results: 
 \bit 
 \item {\bf (Continuity $=$ Sequential continuity)} A mapping $f: X\rightarrow Y$  is continuous if, and only if, it is sequentially continuous, i.e., for each convergent sequence $x_i\rightarrow a$ also the image sequence $f(x_i)$ converges to $f(a)$. 
 \item {\bf (Compact $=$ Sequential compact)} A subset $A\subset X$ (in a metric space $X$) is compact, if and only if, it is sequentially compact, i.e., every sequence in $A$ has a convergent subsequence. 
 \eit We refer to \cite[pp.~103--105]{Jan99} or \cite[Thm.~1.10(b)]{Dobro} in this context.
  \begin{rem}\label{comp-seqcomp}
   Let us point out here that in view of compactness for the 'single' half-normed (and therefore pseudo-metric) spaces $H^k_{|.|}([0,T],\tilde{X})$ the following holds: If $K\subset H^k_{|.|}([0,T],\tilde{X})$ is compact, then it is also sequential compact. This follows from \cite[p.~104]{Jan99}, since   a pseudo-metric space satisfies the first axiom of countability. 
 \end{rem}

\section{Details for the proof of Theorem \ref{nonlin-B-reg1}}
\label{app-proof}

In this appendix we show how to obtain estimate  \eqref{est-ball_a} in the proof of  Theorem \ref{nonlin-B-reg1}: 

We consider the second term 
in \eqref{est-1} and calculate 

\begin{align}
II&= c_{\tilde{L}^{-1}} \sum_{k=l}^{l+1}|u^M-v^M|_{H^{k}([0,T], L_2(D))} \notag\\
&= c_{\tilde{L}^{-1}}\sum_{k=l}^{l+1}\left|(u-v)\sum_{j=0}^{M-1} u^jv^{M-1-j}\right|_{H^{k}([0,T],L_2(D))}\notag\\
&\lesssim c_{\tilde{L}^{-1}}\sum_{k=l}^{l+1}\left\|\partial_{t^k}\left[(u-v)\sum_{j=0}^{M-1} u^jv^{M-1-j}\right]\Big|L_2(D_T)\right\| \notag\\
&= c_{\tilde{L}^{-1}}\sum_{k=l}^{l+1}\left\|\sum_{w=0}^k{k \choose w}\partial_{t^w}(u-v)\cdot \right. 
\left.\left[\sum_{j=0}^{M-1} \sum_{r=0}^{k-w} {{k-w}\choose r} \partial_{t^{r}}u^j \cdot \partial_{t^{k-w-r}}v^{M-1-j}\right]\Big|L_2(D_T)\right\| \notag\\
&\lesssim  c_{\tilde{L}^{-1}}\sum_{k=l}^{l+1}\left\|\sum_{w=0}^k|\partial_{t^w}(u-v)|\cdot \right.
\left.\left[\sum_{j=0}^{M-1} \sum_{r=0}^{k-w}  |\partial_{t^{r}}u^j \cdot \partial_{t^{k-w-r}}v^{M-1-j}|\right]\Big|L_2(D_T)\right\|, 
\label{est-1a}
\end{align}
where we re-used Leibniz's formula from \eqref{est-2} (since this is a pointwise estimate) in the second but last line. Then again  Fa\`{a} di Bruno's formula, cf. \eqref{FaaDiBruno}, is applied  in order to estimate the derivatives in \eqref{est-1a}.

We use the following multiplier result which is a simple consequence of Sobolev's embedding theorem, 
\begin{equation}\label{multiplier-lim}
\|uv|L_2(D)\|\leq \|u|C(D)\|\cdot \|v|L_2(D)\| \leq \|u|H^m(D)\|\cdot \|v|L_2(D)\|,     
\end{equation}
which holds for  $m\geq 2$. 
With this we obtain 
\begin{align}
\Big\|  \partial_{t^r}u^j  | L_2(D)\Big\| 
& \leq  c_{r,j}\left\|\sum_{\kappa_1+\ldots +\kappa_r\leq j} u^{j-(\kappa_1+\ldots +\kappa_r)}\prod_{{i}=1}^r \left|\partial_{t^{{i}}}u\right|^{\kappa_{{i}}}\Big| L_2(D)\right\| \notag\\
&\lesssim \sum_{\kappa_1+\ldots +\kappa_r\leq j} \left\| u| H^m(D)\right\|^{j-(\kappa_1+\ldots +\kappa_r)} \prod_{{i}=1}^{r-1} \left\| \partial_{t^{{i}}}u| H^m(D)\right\|^{\kappa_{{i}}} \left\| \partial_{t^r}u| L_2(D)\right\|^{\kappa_r}.  \label{est-3a}
\end{align}
Similar  for $\partial_{t^{k-w-r}}v^{M-1-j}$. As before, from \eqref{cond-kr}  we observe $\kappa_{r}\leq 1$, therefore the highest derivative $u^{(r)}$ appears at most once.  Note that since $H^m(D)$ is a multiplication algebra for $m> \frac 32$, we get \eqref{est-3a} with $L_2(D)$ replaced by $H^m(D)$ as well.  
Now {\eqref{multiplier-lim} and \eqref{est-3a}  inserted in \eqref{est-1a}} give  
\begin{align}
II&=c_{\tilde{L}^{-1}} \sum_{k=l}^{l+1}|u^M-v^M|_{H^{k}([0,T], L_2(D))}\notag\\
& {
\lesssim  c_{\tilde{L}^{-1}}\sum_{k=0}^{l+1}\Bigg(\int_0^T\left\|\partial_{t^k}(u-v)\sum_{j=0}^{M-1}u^jv^{M-1-j}\Big| L_2(D)\right\|^2 \ud t\Bigg)^{1/2}
}\notag\\
& {
\lesssim c_{\tilde{L}^{-1}}\sum_{k=l}^{l+1}\sum_{w=0}^k\Bigg(\int_0^T\bigg\{\left\|\partial_{t^w}(u-v)|H^m(D)\right\|^2 }\notag\\
& \qquad\qquad  {\sum_{j=0}^{M-1}\sum_{r=0}^{k-w}\left\|\partial_{t^r}u^j|H^m(D)\|^2 \|\partial_{t^{k-w-r}}v^{M-1-j}| H^m(D)\right\|^2 }
\bigg\}\ \ud t\Bigg)^{1/2}\notag \\
&{\lesssim c_{\tilde{L}^{-1}}\sum_{k=l}^{l+1}\sum_{w=0}^k\Bigg(\int_0^T\left\|\partial_{t^w}(u-v)| H^m(D)\right\|^2 \cdot }\notag \\
& {\qquad \qquad \sum_{j=0}^{M-1} \sum_{r=0}^{k-w}\sum_{\kappa_1+\ldots +\kappa_{r}\leq j, \atop \kappa_1+2\kappa_2+\ldots+ r\kappa_{r}\leq r}
  \left\| u| H^m(D)\right\|^{2(j-(\kappa_1+\ldots +\kappa_{r}))}} 
\prod_{{i}=1}^{r} \left\| \partial_{t^{{i}}}u| H^m(D)\right\|^{2\kappa_{{i}}} \notag \\ 
& \qquad   \sum_{\kappa_1+\ldots +\kappa_{k-w-r}\leq M-1-j, \atop \kappa_1+2\kappa_2+\ldots+(k-w-r)\kappa_{k-w-r}\leq k-w-r} 
  \left\| v| H^m(D)\right\|^{2(M-1-j-(\kappa_1+\ldots +\kappa_{k-w-r}))} 
\prod_{{i}=1}^{w-k-r} \left\| \partial_{t^{{i}}}v| H^m(D)\right\|^{2\kappa_{{i}}}
\ud t\Bigg)^{1/2}
\notag\\
&{\lesssim c_{\tilde{L}^{-1}}\sum_{k=l}^{l+1}\Bigg(\int_0^T\left\|\partial_{t^k}(u-v)| H^m(D)\right\|^2 \cdot }\notag \\
& {\qquad  M \sum_{\kappa_1'+\ldots +\kappa_{k}'\leq \min\{M-1,k\}}
  \max_{w\in \{u,v\}}\left\| w| H^m(D)\right\|^{2(M-1-(\kappa_1'+\ldots +\kappa_{k}'))}} 
\max\Big(\prod_{{i}=1}^{k} \left\| \partial_{t^{{i}}}w| H^m(D)\right\|^{4\kappa_{i}'}, 1\Big)
\ud t\Bigg)^{1/2}
\notag\\
&\lesssim c_{\tilde{L}^{-1}}M \sum_{k=l}^{l+1}
|u-v|_{H^{k}([0,T],H^m(D))}\notag \\
&\qquad \max_{w\in \{u,v\}}\max_{{r}=0,\ldots, l+1} \max 
\left(
\left\| \partial_{t^{{r}}}w|L_{\infty}([0,T],H^m(D))\right\|,\   1\right)^{{{2(M-1)}}}.\notag\\ \label{est-4a}
\end{align}
For the re-definition of the $\kappa_i$'s in the second but last line in \eqref{est-4a} we refer to the explanations given after \eqref{est-4}. 
From Theorem \ref{thm-sob-emb} we see that 
\begin{eqnarray}
u,v \in S&\hookrightarrow &H^{r+1}([0,T],X)
 \hookrightarrow  {C}^{r}([0,T],{\mathring{H}^m}(D)), 
\label{est-4aa}
\end{eqnarray}
hence the term  $\max_{r=0,\ldots, l+1}\max(\ldots)^{2(M-1)}$ in \eqref{est-4a}  is bounded  by  
$\max_{r=0,\ldots, l+1}(R_{r+1}+|\tilde{L}^{-1}f|_{H^{{r+1}}([0,T],X)}+\max_{w\in \{u,v\}}\|w|L_2([0,T],X)\|,1)^{2(M-1)}$,  since $u$ and $v$ are  taken from  $S_0$. 

In conclusion we obtain from \eqref{est-4a} {and \eqref{est-4aa}},  
 \begin{align*}
II&= c_{\tilde{L}^{-1}} \sum_{k=l}^{l+1}|u^M-v^M|_{H^{k}([0,T], L_2(D))}\notag \\
& \lesssim 
c_{\tilde{L}^{-1}} M\max_{r=0,\ldots, l+1}(R_{r+1}+|\tilde{L}^{-1}f|_{H^{{r+1}}([0,T],X)}+\max_{w\in \{u,v\}}\|w|L_2([0,T],X)\|,1)^{2(M-1)} \notag \\
 & \qquad 
\cdot \sum_{k=l}^{l+1}
|u-v|_{H^{k}([0,T],H^m(D))}
\end{align*}
which shows  \eqref{est-ball_a}.

\paragraph{Acknowledgement:}   
We want to thank Prof. Dr. Herbert Amann and Prof. Dr. Winfried Sickel  for helpful discussions on the subject.


\begin{thebibliography}{m}



\bibitem{ADN59}
Agmon, S.,   Douglis, A.,  Nierenberg, L. (1959).   
\newblock Estimates near the boundary for solutions of elliptic partial differential equations satisfying general boundary conditions I.  
\newblock{\em Comm. Pure Appl. Math.} {\bf 12}, 623--727.  

\bibitem{AGI08}
Aimar, H.,  G\'omez, I., Iaffei, B. (2008). 
\newblock {Parabolic mean values and maximal estimates for gradients of temperatures.}
\newblock {\em J. Funct. Anal.} {\bf 255}, 1939--1956.

\bibitem{AGI10}
Aimar, H.,  G\'omez, I., Iaffei, B. (2010).
\newblock {On Besov regularity of temperatures.}
\newblock {\em J. Fourier Anal. Appl.} {\bf 16}, 1007--1020.

\bibitem{AG12}
Aimar, H., G\'omez, I. (2012).
\newblock {Parabolic Besov regularity for the heat equation.}
\newblock {\em Constr. Approx.} {\bf 36}, 145--159.

\bibitem{Ama19}
Amann, H. (2019).   
\newblock {\em Linear and quasilinear parabolic problems, Vol. II: Function spaces}.
\newblock{Monographs in Mathematics}, Vol. 106, Birkh\"{a}user/Springer, Cham. 

\bibitem{AH08}
 Anh, N. T. and Hung, N. M. (2008).
\newblock Regularity of solutions of initial-boundary value problems for parabolic equations in domains with conical points.
\newblock {\em J. Differential Equations} {\bf 245},  no. 7, 1801--1818.



\bibitem{BG97}
Babuska, I., Guo, B. (1997).
\newblock Regularity of the solutions for elliptic problems on nonsmooth domains in $\real^3$, Part I: countably normed spaces on polyhedral domains.
\newblock {\em Proc. Roy. Soc. Edinburgh Sect. A}, {\bf 127}, 77--126.


\bibitem{BMNZ}
{Bacuta, C.,  Mazzucato, A., Nistor, V., and  Zikatanov, L.} (2010). 	
{Interface and mixed boundary value problems on $n$-dimensional polyhedral domains}.
{\em  Doc. Math.} {\bf 15}, 687--745.




\bibitem{CW20}
Cioica-Licht, P. and Weimar, M. (2020).
\newblock {On the limit regularity in Sobolev and Besov scales related to approximation theory.}
\newblock {\em J. Fourier Anal. Appl.} {\bf 26}(1), 10.



\bibitem{Cost19}
Costabel, M. (2019).
\newblock {On the limit Sobolev regularity for Dirichlet and Neumann problems on Lipschitz domains.}
\newblock {\em Math. Nachr.} {\bf 292}, 2165--2173.



 \bibitem{Dah98} 
Dahlke, S. (1998).  
\newblock Besov regularity for elliptic boundary value problems with variable coefficients. 
\newblock {\em Manuscripta Math.} {\bf 95}, 59--77.



\bibitem{Dah99a}
Dahlke, S. (1999).  
\newblock Besov regularity for interface problems. 
\newblock {\em Z. Angew. Math. Mech.} {\bf 79}(6),  
       383--388.     
       
 \bibitem{Dah99b}
Dahlke, S. (1999).  
\newblock Besov regularity for elliptic boundary value problems on polygonal domains.  
\newblock {\em Appl. Math.  Lett.} {\bf 12}(6),  31--38.

\bibitem{Dah02} 
Dahlke, S. (2002).  
\newblock Besov regularity of edge singularities for the Poisson equation in polyhedral
domains.  
\newblock {\em Num. Linear Algebra  Appl.} {\bf 9}(6--7), 457--466. 


\bibitem{DDD}
Dahlke, S., Dahmen, W.,  DeVore, R. (1997). 
\newblock {Nonlinear approximation and adaptive techniques for solving elliptic operator equations}.  
\newblock {\em Multiscale Wavelet Methods for Partial Differential Equations, (W. Dahmen, A.J. Kurdila, and P. Oswald, eds), Wavelet Analysis and Applications}, vol. 6, Academic Press, San Diego,  237-283.  



\bibitem{DDV97}
Dahlke, S.,  DeVore, R. (1997). 
\newblock Besov regularity for elliptic boundary value problems. 
\newblock {\em Comm.  Partial Differential Equations} {\bf 22}(1-2),   1--16. 


\bibitem{DDHSW}
Dahlke, S., Diening, L., Hartmann, C., Scharf, B., Weimar, M. (2016).  
\newblock Besov regularity of solutions to the p--Poisson equation.  
\newblock {\em Nonlinear Anal.} {\bf 130},  298--329, 2016. 



\bibitem{DHSS20}
Dahlke, S., Hansen, M., Schneider, C., and Sickel, W. (2020). 
\newblock Properties of Kondratiev spaces.
\newblock {\em Preprint-Reihe Philipps University Marburg}, Bericht Mathematik Nr. 2018-06. 
\newblock { (arXiv:1911.01962)}






\bibitem{DS19}
Dahlke, S. and Schneider, C. (2019).   
\newblock Besov regularity of parabolic and hyperbolic PDEs. 
\newblock {\em Anal. Appl.} {\bf 17}, no. 2,  235--291. 



\bibitem{DS20}
Dahlke, S. and Schneider, C. (2021).   
\newblock Regularity in Sobolev and Besov spaces for parabolic problems on domains of polyhedral type. 
\newblock{\em To appear in J. Geom. Anal.}. 


\bibitem{DL93}
DeVore, R.A. and  Lorentz, G.G. (1993)
\newblock {\em Constructive approximation}.
\newblock Grundlehren der Mathematischen Wissenschaften {303},
\newblock Springer, Berlin. 




\bibitem{Dobro}
Dobrowolski M. (2010).   
\newblock {\em Angewandte Funktionalanalysis}.
\newblock{Springer-Lehrbuch Masterclass}, Springer, Berlin. 


\bibitem{Eva10}
Evans, L. C.  (2010). 
\newblock Partial differential equations.
\newblock {\em Graduate Studies in Mathematics} {\bf 19}, 2nd edition, American Mathematical Society, Providence, RI. 




\bibitem{Gris92}
Grisvard, P. (1992).
\newblock {\em Singularities in boundary value problems.}
\newblock {Recherches en math\'ematiques appliqu\'ees}, vol. 22, Masson, Paris;  Springer-Verlag, Berlin. 


\bibitem{Gris11}
Grisvard, P. (2011).
\newblock {\em Elliptic problems in nonsmooth domains.}
\newblock Reprint of the 1985 original. { Classics in Applied Mathematics}, 69, SIAM, Philadelphia. 



\bibitem{Han15}
Hansen, M. (2015).
\newblock Nonlinear approximation rates and Besov regularity for elliptic PDEs on polyhedral domains.
\newblock {\em Found. Comput. Math.} {\bf 15}, 561--589.

\bibitem{HW18}
Hartmann, C. and Weimar, M. (2018).
\newblock Besov regularity of solutions to the $p$-Poisson equation in the vicinity of a vertex of a polygonal domain.
\newblock {\em Results Math.} {\bf 73}(41), 1--28.


\bibitem{Jan99}
K. J\"anich (1999).
\newblock {\em Topologie.}
\newblock Springer-Verlag, Berlin. 

\bibitem{JK95}
Jerison, D., Kenig, C.E. (1995).
\newblock The inhomogeneous Dirichlet problem in Lipschitz domains.
\newblock {\em J. Funct. Anal.} {\bf 130}, 161--219. 



\bibitem{KMR01}
Kozlov, V.A., Mazya, V.G., and Rossman, J. (2001).
\newblock Spectral problems associated with corner singularities of solutions to elliptic equations.  
\newblock {\em Mathematical Surveys and Monographs},  {\bf 85}, American Mathematical Society, Providence, RI. 

\bibitem{Kre18}
Kreuter, M. (2018)
\newblock {\em Vector-valued elliptic boundary value problems on rough domains.}  
\newblock PhD thesis, University of Ulm. 

\bibitem{Lan01}
Lang, J. (2001). 
\newblock {\em Adaptive multilevel solution of nonlinear parabolic {PDE} systems}. 
\newblock {Lecture Notes in Computational Science and Engineering; Theory, algorithm, and applications}, vol. 16, Springer-Verlag, Berlin. 



\bibitem{LL15}
Luong, V.T., Loi, D.V. (2015).
\newblock The first initial-boundary value problem for parabolic equations in a cone with edges.
\newblock {\em Vestn. St.-Petersbg. Univ. Ser. 1. Mat. Mekh. Asron.  2}, {\bf 60}(3), 394--404.



\bibitem{MR10}
Maz'ya, V.,  Rossmann, J. (2010).
\newblock{\em Elliptic equations in polyhedral domains}. 
\newblock{Mathematical Surveys and Monographs}, vol. 162, American Mathematical Society. 


\bibitem{NistorMazzucato}
{Mazzucato, A. and  Nistor, V.} (2010). 
\newblock {Well posedness and regularity for the elasticity equation with mixed boundary conditions on
			polyhedral domains and domains with cracks}. 
\newblock{\em  Arch. Ration. Mech. Anal.} {\bf 195}, 25--73.




\bibitem{MN75}
Morris, S. A. and Noussair, E. S. (1975).   
\newblock The Schauder-Tychonoff Fixed Point Theorem and Applications.
\newblock{\em Matematick\'y \v{c}asopsis} {\bf 25}, 165--172.   {\em URL:  http://dml.cz/dmlcz/126953}

\bibitem{Rud76}
Rudin, W. (1976).
\newblock {\em Principles of mathematical analysis}.
\newblock {International Series in Pure and Applied Mathematics}, McGraw-Hill Book Co., New York-Auckland-D\"{u}sseldorf. 


\bibitem{RS96}
Runst, T., Sickel, W. (1996).
\newblock {\em Sobolev spaces of fractional order, {N}emytskij operators, and
              nonlinear partial differential equations}.
\newblock {De Gruyter Series in Nonlinear Analysis and Applications}. 




\bibitem{co-buch}
 Schneider, C. (2021). 
 \newblock  Beyond Sobolev and Besov:  Regularity of solutions of PDEs and their  traces in function spaces. 
 \newblock {\em LNM}, {\bf 2291}, Springer. 


\bibitem{co-habil}
 Schneider, C. (2020). 
 \newblock  Besov regularity of partial differential equations, and traces in function spaces. 
 \newblock {\em Habilitationsschrift}, Philipps-Universit\"at Marburg. 



\bibitem{SS09} 
Stevenson, R.,  Schwab, C. (2009). 
\newblock Space-time adaptive wavelet methods for parabolic evolution problems.  
\newblock {\em Math. Comput.},  {\bf 78}, 1293--1318. 


\bibitem{Tho06}
Thom\'e{}e, V. (2006). 
\newblock {\em Galerkin finite element methods for parabolic problems}. 
\newblock  {
Springer Series in Computational Mathematics}, vol. 25, Springer-Verlag, Berlin, 2nd edition. 

\bibitem{Tri83}
Triebel, H. (1983).
\newblock {\em Theory of Function spaces}.
\newblock { Monographs in Mathematics}, vol. 78, Birkh\"auser Verlag, Basel. 


\bibitem{Tri08}
Triebel, H. (2008).
\newblock {\em Function spaces and wavelets on domains}.
\newblock {EMS Tracts on mathematics 7}, EMS Publishing House, Z\"urich. 






\end{thebibliography}
\end{document}